\documentclass[psamsfonts, 12pt,reqno]{amsart}

\usepackage[backref]{hyperref}
\usepackage{amssymb}
 \numberwithin{equation}{section}
\usepackage{graphicx}
\usepackage{amsmath}
\usepackage{mathrsfs}
\usepackage{latexsym}
\usepackage{mathrsfs}
\usepackage{graphicx}
\usepackage[utf8]{inputenc}
\usepackage[T1]{fontenc}

\usepackage{color}



\theoremstyle{plain}
\newtheorem{theorem}{Theorem}[section]
\newtheorem{proposition}{Proposition}[section]
\newtheorem{lemma}{Lemma}[section]

\newtheorem{rem}{Remark}[section]

\begin{document}
\title[]{Normalized Semiclassical Solutions to Magnetic Schr\"{o}dinger-Poisson Systems with Critical Local and Nonlocal Interactions}

\author{Khaled Khachnaoui}

%\address{ Khaled Khachnaoui\\
 %University of Kairouan, Preparatory Institute for Engineering Studies \\Department of
 %Mathematics\\Tunisia.}\email{k{\_}khachnaoui@yahoo.com}
\begin{abstract}
We study the existence, multiplicity, and concentration of
normalized semiclassical states for a magnetic
Schr\"odinger--Poisson system in $\mathbb{R}^3$ featuring both the
Sobolev-critical local nonlinearity $|u|^4u$ and a critical nonlocal
Poisson interaction. The problem is considered under the prescribed
mass constraint $\int_{\mathbb{R}^3}|u|^2\,dx=a^2\varepsilon^3,$
where $a>0$ denotes the prescribed mass and $\varepsilon>0$ is the
semiclassical parameter. By combining constrained variational
methods, a suitable penalization scheme, concentration--compactness
arguments, and Ljusternik--Schnirelmann theory, we first prove the
existence of a normalized semiclassical solution for sufficiently
small $a$ and $\varepsilon$. We then establish a multiplicity result
showing that, for every sufficiently small $\varepsilon>0$, the
number of distinct normalized solutions is bounded from below by the
Ljusternik--Schnirelmann category of the minimum set
\[
\mathcal M = \{x\in\mathbb{R}^3:V(x)=\min_{\mathbb{R}^3}V\}.
\]
Finally, we describe the semiclassical concentration phenomenon by
showing that the maximum points of the resulting solutions approach
$\mathcal M$ as $\varepsilon\to0$.
\end{abstract}

\maketitle

\tableofcontents
\section{Introduction}
\label{sec:introduction}

In this paper, we investigate the existence, multiplicity, and
concentration behavior of normalized semiclassical solutions for the
following magnetic Schr\"odinger--Poisson system:
\begin{equation}
\label{P}
\begin{cases}
(-i\varepsilon\nabla-A(x))^2u+V(x)u-\phi |u|^3u = \lambda u+\mu
|u|^{q-2}u+|u|^4u,
& x\in\mathbb R^3,\\[2mm]
-\varepsilon^2\Delta\phi=|u|^5, & x\in\mathbb R^3,
\end{cases}
\end{equation}
subject to the prescribed mass constraint
\begin{equation}
\label{constraint} \int_{\mathbb R^3}|u|^2\,dx = a^2\varepsilon^3.
\end{equation}
Here, $a>0, \qquad \mu>0, \qquad 2<q<\frac{10}{3},$ and
$\varepsilon>0$ is the semiclassical parameter. The function
$A:\mathbb R^3\to\mathbb R^3$ denotes the magnetic potential, while
$V:\mathbb R^3\to\mathbb R$ is the electric potential. The unknown
$u:\mathbb R^3\to\mathbb C$ represents the wave function, whereas
$\phi:\mathbb R^3\to\mathbb R$ is the electrostatic potential
generated by the charge density $|u|^5$. The parameter
$\lambda\in\mathbb R$ is not prescribed in advance; it arises as a
Lagrange multiplier associated with the mass constraint
\eqref{constraint}.

Schr\"odinger--Poisson systems provide a fundamental mathematical
model for the interaction between a quantum particle and the
electrostatic field generated by its own charge distribution. Such
systems arise naturally in quantum mechanics, semiconductor theory,
plasma physics, and nonlinear optics. The classical
Schr\"odinger--Poisson model was introduced by Benci and Fortunato
\cite{BenciFortunato} and has subsequently attracted considerable
attention. We refer, among others, to
\cite{AmbrosettiRuiz,DAprileMugnai,AzzolliniDAprize,CeramiVaira} and
the references therein.

The presence of an external magnetic field substantially modifies
the analytical and variational structure of the problem. In this
case, the usual gradient is replaced by the semiclassical magnetic
gradient $\nabla_{\varepsilon,A}u := -i\varepsilon\nabla u-A(x)u,$
and the corresponding magnetic Schr\"odinger operator is
$(-i\varepsilon\nabla-A(x))^2.$ This operator describes the kinetic
energy of a charged quantum particle moving under the influence of
an external magnetic field. As a consequence, the wave function is
complex-valued and the natural variational framework is a magnetic
Sobolev space. Magnetic nonlinear Schr\"odinger equations have been
extensively studied since the pioneering contributions of Esteban
and Lions \cite{EstebanLions}; see also Cingolani and Secchi
\cite{CingolaniSecchi} and the references therein.

A further distinctive feature of \eqref{P} is the prescribed mass
condition \eqref{constraint}. In the standard unconstrained
variational setting, the frequency is fixed and the mass of the
solution is determined a posteriori. By contrast, in the normalized
framework, the $L^2$-mass is prescribed in advance, while the
frequency $\lambda$ becomes an unknown parameter that must be
determined together with the solution. Thus, normalized solutions
are obtained as constrained critical points of the associated energy
functional on an $L^2$-sphere. This variational viewpoint is closely
related to the classical works of Lions \cite{Lions1} and Cazenave
\cite{Cazenave} and has subsequently been developed in many
directions; see, for instance, Bartsch and de Valeriola
\cite{BartschValeriola}.

In the semiclassical regime, the mass normalization in
\eqref{constraint} is naturally scaled by the factor
$\varepsilon^3$. Indeed, a state concentrating at the scale
$\varepsilon$ around a point $y\in\mathbb R^3$ has the form
$u_\varepsilon(x) = w\left(\frac{x-y}{\varepsilon}\right),$ and
therefore
\[
\int_{\mathbb R^3}|u_\varepsilon|^2\,dx = \varepsilon^3
\int_{\mathbb R^3}|w|^2\,dx.
\]
Thus, the constraint $\|u_\varepsilon\|_2^2=a^2\varepsilon^3$ is
consistent with an autonomous limiting profile having prescribed
mass $a$.

Considerable progress has been made in the study of critical
Schr\"odinger-Poisson equations. Feng \cite{FX} investigated the
existence and concentration of positive ground state solutions for
the system
\begin{equation}
\label{fcv}
\begin{cases}
-\varepsilon^2\Delta u+V(x)u-\phi|u|^3u = f(u)+|u|^4u,
& x\in\mathbb R^3,\\[2mm]
-\varepsilon^2\Delta\phi=|u|^5, & x\in\mathbb R^3,
\end{cases}
\end{equation}
where the perturbation $f$ has Sobolev-subcritical growth. By
combining a modified concentration--compactness principle with the
mountain-pass theorem, the author obtained positive ground states
and described their concentration behavior in the semiclassical
limit.

The normalized counterpart of critical Schr\"odinger--Poisson
systems has also received increasing attention. In \cite{VR}, the
authors established existence and multiplicity results for
normalized semiclassical solutions to a Sobolev-critical
Schr\"odinger--Poisson problem involving a critical nonlocal
interaction. By means of Ljusternik--Schnirelmann theory, the number
of concentrating solutions was related to the topology of the set
where the electric potential attains its minimum.

Several contributions have also been devoted to
Schr\"odinger--Poisson equations with prescribed mass. Bartsch and
Jeanjean \cite{BartschJeanjean} studied normalized solutions in the
presence of critical nonlinear effects, while multiplicity phenomena
for related non-magnetic problems were investigated by Jeanjean and
Lu \cite{JeanjeanLu} and by Bellazzini, Jeanjean, and Luo
\cite{BellazziniJeanjeanLuo}.

A particularly relevant contribution was made by Meng and He
\cite{MY}, who considered the critical normalized system
\begin{equation}
\label{10J}
\begin{cases}
-\Delta u-\phi|u|^3u = \lambda u+\mu|u|^{q-2}u+|u|^4u,
& x\in\mathbb R^3,\\[2mm]
-\Delta\phi=|u|^5,
& x\in\mathbb R^3,\\[2mm]
\displaystyle \int_{\mathbb R^3}|u|^2\,dx=a^2.
\end{cases}
\end{equation}
Using genus theory, they obtained multiplicity results in the
$L^2$-subcritical regime and also studied an $L^2$-supercritical
perturbation when the parameter $\mu>0$ is sufficiently large.

Subsequently, He, Meng, and R\u{a}dulescu \cite{HXX} further
investigated problem \eqref{10J} with the perturbation
$\mu|u|^{q-2}u$, covering the $L^2$-subcritical, $L^2$-critical, and
$L^2$-supercritical regimes. They established several existence and
nonexistence results and analyzed the asymptotic behavior of ground
states as $\mu\to0^+$.

More recently, He, Liu, and Meng \cite{HXL} studied the generalized
normalized problem
\begin{equation}
\label{potential_problem}
\begin{cases}
-\Delta u+V(x)u-\phi|u|^3u = \lambda u+\mu|u|^{q-2}u+|u|^4u,
& x\in\mathbb R^3,\\[2mm]
-\Delta\phi=|u|^5,
& x\in\mathbb R^3,\\[2mm]
\displaystyle \int_{\mathbb R^3}|u|^2\,dx=a^2,
\end{cases}
\end{equation}
where the external potential $V$ vanishes at infinity. Their
approach combines the Pohozaev manifold method, constrained
minimization, and delicate variational estimates. Their analysis
covers the $L^2$-subcritical range $2<q<\frac{10}{3},$ the
$L^2$-critical case $q=\frac{10}{3},$ and the $L^2$-supercritical
regime $\frac{10}{3}<q<6.$ They also established decay properties of
the corresponding ground states.

Despite these important advances, the existing literature on
normalized critical Schr\"odinger--Poisson equations is largely
concerned with non-magnetic problems. The simultaneous presence of a
magnetic field, a prescribed mass constraint, a Sobolev-critical
local nonlinearity, a critical nonlocal Poisson interaction, and a
semiclassical parameter leads to a substantially different
variational structure. In particular, the extension of the available
non-magnetic theory to the present setting is not straightforward.

We now describe the main analytical difficulties associated with
problem \eqref{P}.

First, the magnetic potential $A(x)$ forces the solutions to be
complex-valued. Consequently, the variational analysis must be
carried out in the semiclassical magnetic Sobolev space
$H_{\varepsilon,A}^1(\mathbb R^3,\mathbb C).$ Several arguments
based on positivity, rearrangements, or pointwise ordering, which
are available for real-valued problems, cannot be applied directly.
A fundamental tool in overcoming this difficulty is the diamagnetic
inequality \cite{Diamagnetic}, which allows one to compare the
modulus of a magnetic function with functions in the usual real
Sobolev space.

Second, eliminating the electrostatic potential through the Poisson
equation gives
\[
\phi_u(x) = \frac1{4\pi\varepsilon^2} \int_{\mathbb R^3}
\frac{|u(y)|^5}{|x-y|}\,dy.
\]
Therefore, the system contains the critical nonlocal interaction
\[
\int_{\mathbb R^3} \phi_u|u|^5\,dx = \frac1{4\pi\varepsilon^2}
\iint_{\mathbb R^3\times\mathbb R^3} \frac{|u(x)|^5|u(y)|^5}{|x-y|}
\,dx\,dy.
\]
The treatment of this term requires the Hardy--Littlewood--Sobolev
inequality and suitable nonlocal splitting and convergence
arguments.

Third, the local term $|u|^4u$ has critical Sobolev growth in
dimension three. The embedding $H^1(\mathbb R^3) \hookrightarrow
L^6(\mathbb R^3)$ is continuous but not compact. Hence Palais--Smale
sequences may lose compactness through translation, concentration,
or splitting. In the present problem, this phenomenon is more
delicate because the local Sobolev-critical term interacts with a
second critical mechanism generated by the Poisson coupling.

Fourth, the prescribed mass constraint creates an additional
variational difficulty. The associated energy functional is
unbounded from below on the whole mass sphere because of the
critical terms. Hence direct global minimization is not available.
It is necessary to identify an appropriate local variational
structure, construct suitable minimax levels, and establish
compactness below the first critical concentration threshold.

Finally, the semiclassical limit requires the localization of
solutions near the minimum set of the electric potential. This is
particularly delicate in the magnetic setting, since localized
profiles must incorporate a suitable magnetic phase. Moreover, the
critical nonlocal interaction must be controlled simultaneously with
the local critical term throughout the concentration process.

To address these difficulties, we combine constrained variational
methods, a suitable penalization procedure,
concentration--compactness arguments, the diamagnetic inequality,
and Ljusternik--Schnirelmann category theory. The penalization
method, inspired by the strategy of del Pino and Felmer
\cite{delPinoFelmer}, is used to recover compactness and to force
low-energy states to concentrate inside the potential well. The
concentration--compactness principle, in the spirit of Lions
\cite{Lions2,Lions3}, is employed to exclude vanishing, splitting,
and critical bubbling below suitable energy levels. Finally, a
barycenter construction connects the topology of the minimum set of
$V$ with the multiplicity of normalized semiclassical solutions.

We now state the assumptions on the electric and magnetic
potentials. Let $V_0 := \inf_{x\in\mathbb R^3}V(x).$ We assume that:
\begin{itemize}
\item[$(V1)$]
$V\in C(\mathbb R^3,\mathbb R)$ and $V_0>0.$
\item[$(V2)$]
There exists a bounded open set $\Lambda\subset\mathbb R^3$ such
that $V_0 < \min_{x\in\partial\Lambda}V(x).$

\item[$(V3)$]
The minimum set
\[
\mathcal M := \left\{ x\in\Lambda: V(x)=V_0 \right\}
\]
is nonempty.
\end{itemize}
For the magnetic potential, we assume that:
\begin{itemize}
\item[$(A)$]
$A\in C(\mathbb R^3,\mathbb R^3)$ and there exists $C_A>0$ such that
\[
|A(x)| \leq C_A(1+|x|) \qquad \text{for every }x\in\mathbb R^3.
\]
\end{itemize}
Under these assumptions, we establish the existence of normalized
semiclassical solutions for sufficiently small masses and
semiclassical parameters.

\begin{theorem}\label{T1}
Assume that $(V1)$--$(V3)$ and $(A)$ hold. Then there exist $a_*>0$
and $\varepsilon_*>0$ such that, for every $a\in(0,a_*),$ and
$\varepsilon\in(0,\varepsilon_*),$ problem \eqref{P} admits at least
one normalized semiclassical solution $u_\varepsilon \in
H_{\varepsilon,A}^1(\mathbb R^3,\mathbb C)$ satisfying
$\int_{\mathbb R^3} |u_\varepsilon|^2\,dx = a^2\varepsilon^3.$ More
precisely, there exists $\lambda_\varepsilon\in\mathbb R$ such that
the pair $(u_\varepsilon,\lambda_\varepsilon)$ satisfies the first
equation of \eqref{P} in the weak sense, with the electrostatic
potential determined by $-\varepsilon^2\Delta\phi_{u_\varepsilon} =
|u_\varepsilon|^5 \qquad \text{in }\mathbb R^3.$
\end{theorem}

Our second main result establishes a multiplicity phenomenon
governed by the topology of the minimum set $\mathcal M$ and
describes the concentration of the corresponding semiclassical
states. For $\delta>0$, define
\[
\mathcal M_\delta := \left\{ x\in\mathbb R^3:
\operatorname{dist}(x,\mathcal M)\leq\delta \right\}.
\]
\begin{theorem}\label{T2}
Assume that $(V1)$--$(V3)$ and $(A)$ hold. Let $a\in(0,a_*)$ be
fixed and let $\delta>0$ be sufficiently small so that $\mathcal
M_\delta\subset\Lambda.$ Then there exists $\varepsilon_0 =
\varepsilon_0(a,\delta)>0$ such that, for every
$\varepsilon\in(0,\varepsilon_0),$ problem \eqref{P} admits at least
$\operatorname{cat}_{\mathcal M_\delta}(\mathcal M)$ distinct
normalized semiclassical solutions. More precisely, there exist
distinct solutions $u_{\varepsilon,1}, \ldots, u_{\varepsilon,\ell},
\qquad \ell \geq \operatorname{cat}_{\mathcal M_\delta}(\mathcal
M),$ satisfying $\int_{\mathbb R^3} |u_{\varepsilon,j}|^2\,dx =
a^2\varepsilon^3, \qquad j=1,\ldots,\ell.$ Furthermore, let
$\varepsilon_n\to0$ and let $u_{\varepsilon_n}$ be any sequence
selected from the families of solutions obtained above. If
$x_{\varepsilon_n}$ is a global maximum point of
$|u_{\varepsilon_n}|$, then
\[
\operatorname{dist} (x_{\varepsilon_n},\mathcal M) \to0 \qquad
\text{as }n\to\infty.
\]
Consequently, $V(x_{\varepsilon_n}) \to V_0.$ Thus, the
corresponding normalized semiclassical states concentrate around the
minimum set $\mathcal M$ as $\varepsilon_n\to0$.
\end{theorem}

The novelty of the present work lies in the simultaneous treatment
of several competing features: the magnetic structure, the
prescribed mass constraint, the Sobolev-critical local nonlinearity,
the critical nonlocal Poisson interaction, and the semiclassical
concentration mechanism. The complex-valued nature of the solutions
requires a magnetic variational framework together with the
diamagnetic inequality, whereas the two critical terms generate
distinct but interacting mechanisms of loss of compactness. An
additional difficulty is caused by the mass constraint, since the
Lagrange multiplier is not prescribed and the energy functional is
not globally bounded from below on the mass sphere.

Our approach first identifies a local ground-state level for the
autonomous limiting problem. We then introduce a penalized
functional and establish its constrained variational geometry.
Localized magnetic test functions yield sharp semiclassical upper
estimates for the relevant minimax levels. A compactness analysis
below the critical concentration threshold produces low-energy
critical points of the penalized functional. Their concentration
near $\mathcal M$ and uniform decay away from the concentration
region imply that the penalization becomes inactive for sufficiently
small $\varepsilon$. Finally, a barycenter map and
Ljusternik--Schnirelmann theory provide the multiplicity result.

To the best of our knowledge, normalized semiclassical solutions for
a magnetic Schr\"odinger--Poisson system combining simultaneously a
prescribed $L^2$-mass, a Sobolev-critical local term, and a critical
nonlocal Poisson interaction have not previously been studied in
this form.

The remainder of the paper is organized as follows. In
Section~\ref{sec:functional_setting}, we introduce the magnetic
functional framework, reduce the Poisson equation, and formulate the
constrained variational problem. In Section \ref{sec:main_proofs},
we study the autonomous limit problem, introduce the penalization
scheme, establish the required compactness properties, and prove
Theorems \ref{T1} and \ref{T2}.
\section{Analytical Tools and Functional Setting}
\label{sec:functional_setting} In this section, we introduce the
magnetic functional framework and collect the preliminary analytical
tools required throughout the paper. We also eliminate the
electrostatic potential and formulate problem \eqref{P} as a
constrained variational problem.
\subsection{The magnetic Sobolev space}
Let $A:\mathbb R^3\to\mathbb R^3$ be the magnetic potential and let
$\varepsilon>0$. For a complex-valued function $u:\mathbb
R^3\to\mathbb C,$ we define the semiclassical magnetic gradient by
\[
\nabla_{\varepsilon,A}u := \left( \frac{\varepsilon}{i}\nabla-A(x)
\right)u = -i\varepsilon\nabla u-A(x)u.
\]
The associated magnetic Schr\"odinger operator is $\left(
\frac{\varepsilon}{i}\nabla-A(x) \right)^2.$ We define the
semiclassical magnetic Sobolev space by
\[
H_{\varepsilon,A}^1(\mathbb R^3,\mathbb C) := \left\{ u\in
L^2(\mathbb R^3,\mathbb C): \nabla_{\varepsilon,A}u \in L^2(\mathbb
R^3,\mathbb C^3) \right\}.
\]
Under assumption $(V1)$, we equip this space with the norm
\[
\|u\|_{\varepsilon,A}^2 := \int_{\mathbb R^3} \left(
|\nabla_{\varepsilon,A}u|^2 + V(x)|u|^2 \right)\,dx.
\]
Since $V(x)\geq V_0>0,$ this norm controls the $L^2$-norm. The
associated real scalar product is
\[
\langle u,v\rangle_{\varepsilon,A} := \mathfrak{Re} \int_{\mathbb
R^3} \left( \nabla_{\varepsilon,A}u \cdot
\overline{\nabla_{\varepsilon,A}v} + V(x)u\overline v \right)\,dx.
\]
Thus, $H_{\varepsilon,A}^1(\mathbb R^3,\mathbb C)$ is regarded as a
real Hilbert space. A basic tool in the magnetic setting is the
diamagnetic inequality.
\begin{lemma}[Diamagnetic inequality]
\label{lem:diamag} For every $u\in H_{\varepsilon,A}^1(\mathbb
R^3,\mathbb C),$ one has
\[
\varepsilon |\nabla |u|(x)| \leq |\nabla_{\varepsilon,A}u(x)| \qquad
\text{for a.e. }x\in\mathbb R^3.
\]
Consequently, $|u|\in H^1(\mathbb R^3,\mathbb R)$ and
\[
\varepsilon^2 \int_{\mathbb R^3} |\nabla |u||^2\,dx \leq
\int_{\mathbb R^3} |\nabla_{\varepsilon,A}u|^2\,dx.
\]
\end{lemma}

As a consequence of Lemma \ref{lem:diamag} and the classical Sobolev
inequality, for every fixed $\varepsilon>0$,
\[
H_{\varepsilon,A}^1(\mathbb R^3,\mathbb C) \hookrightarrow
L^r(\mathbb R^3,\mathbb C), \qquad 2\leq r\leq6.
\]
More precisely, for every $r\in[2,6]$, there exists a constant
$C_{r,\varepsilon}>0$ such that $\|u\|_r \leq C_{r,\varepsilon}
\|u\|_{\varepsilon,A}.$ In particular, $\|u\|_6 \leq
\frac{C}{\varepsilon} \|u\|_{\varepsilon,A}.$ We shall also
repeatedly use the semiclassical Gagliardo-Nirenberg inequality. For
every $r\in[2,6],$ there exists $C_r>0$ such that
\[
\|u\|_r^r \leq C_r \|u\|_2^{(1-\theta_r)r} \|\nabla
|u|\|_2^{\theta_r r},
\]
where $\theta_r = 3\left( \frac12-\frac1r \right).$ Hence, by the
diamagnetic inequality,
\[
\|u\|_r^r \leq C_r \varepsilon^{-\theta_r r} \|u\|_2^{(1-\theta_r)r}
\|u\|_{\varepsilon,A}^{\theta_r r}.
\]
\subsection{Reduction of the Poisson equation}
We now eliminate the electrostatic potential from the system.

Let $D^{1,2}(\mathbb R^3) := \overline{ C_c^\infty(\mathbb R^3)
}^{\,\|\nabla\cdot\|_2}.$ For $u\in H_{\varepsilon,A}^1(\mathbb
R^3,\mathbb C),$ consider the Poisson equation
\begin{equation}
\label{Poisson_auxiliary} -\varepsilon^2\Delta\phi = |u|^5 \qquad
\text{in }\mathbb R^3.
\end{equation}

\begin{lemma}[Reduction of the Poisson equation]
\label{lem:poisson_reduction} For every
\[
u\in H_{\varepsilon,A}^1(\mathbb R^3,\mathbb C),
\]
there exists a unique $\phi_u\in D^{1,2}(\mathbb R^3)$ satisfying
\[
\varepsilon^2 \int_{\mathbb R^3} \nabla\phi_u\cdot\nabla v\,dx =
\int_{\mathbb R^3} |u|^5v\,dx
\]
for every $v\in D^{1,2}(\mathbb R^3).$ Moreover,
\[
\phi_u(x) = \frac1{4\pi\varepsilon^2} \int_{\mathbb R^3}
\frac{|u(y)|^5}{|x-y|}\,dy,
\]
and the following properties hold:
\begin{enumerate}
\item
\[
\phi_u\geq0 \qquad \text{a.e. in }\mathbb R^3;
\]

\item
\[
\varepsilon^2 \int_{\mathbb R^3} |\nabla\phi_u|^2\,dx =
\int_{\mathbb R^3} \phi_u|u|^5\,dx;
\]

\item
\[
\int_{\mathbb R^3} \phi_u|u|^5\,dx = \frac1{4\pi\varepsilon^2}
\iint_{\mathbb R^3\times\mathbb R^3} \frac{ |u(x)|^5|u(y)|^5 }{
|x-y| } \,dx\,dy;
\]

\item
there exists $C>0$, independent of $u$, such that
\[
\int_{\mathbb R^3} \phi_u|u|^5\,dx \leq
C\varepsilon^{-2}\|u\|_6^{10} \leq C\varepsilon^{-12}
\|u\|_{\varepsilon,A}^{10};
\]

\item
for every $t\geq0$,
\[
\phi_{tu} = t^5\phi_u.
\]
\end{enumerate}
\end{lemma}

\begin{proof}
By Lemma \ref{lem:diamag} and the Sobolev embedding,
\[
|u|\in H^1(\mathbb R^3) \hookrightarrow L^6(\mathbb R^3).
\]
Therefore,
\[
|u|^5 \in L^{6/5}(\mathbb R^3) = \bigl(D^{1,2}(\mathbb R^3)\bigr)'.
\]
The Lax--Milgram theorem then yields a unique $\phi_u\in
D^{1,2}(\mathbb R^3)$ such that
\[
\varepsilon^2 \int_{\mathbb R^3} \nabla\phi_u\cdot\nabla v\,dx =
\int_{\mathbb R^3} |u|^5v\,dx
\]
for every $v\in D^{1,2}(\mathbb R^3)$. The Newton potential
representation gives
\[
\phi_u(x) = \frac1{4\pi\varepsilon^2} \int_{\mathbb R^3}
\frac{|u(y)|^5}{|x-y|}\,dy,
\]
and therefore $\phi_u\geq0 \qquad \text{a.e. in }\mathbb R^3.$
Taking $v=\phi_u$ in the weak formulation yields
\[
\varepsilon^2 \int_{\mathbb R^3} |\nabla\phi_u|^2\,dx =
\int_{\mathbb R^3} \phi_u|u|^5\,dx.
\]
Moreover, the Newton representation formula gives
\[
\int_{\mathbb R^3} \phi_u|u|^5\,dx = \frac1{4\pi\varepsilon^2}
\iint_{\mathbb R^3\times\mathbb R^3} \frac{ |u(x)|^5|u(y)|^5 }{
|x-y| } \,dx\,dy.
\]
By the Hardy--Littlewood--Sobolev inequality,
\[
\begin{aligned}
\int_{\mathbb R^3} \phi_u|u|^5\,dx &\leq C\varepsilon^{-2}
\bigl\||u|^5\bigr\|_{6/5}^2
\\
&= C\varepsilon^{-2} \|u\|_6^{10}.
\end{aligned}
\]
Using Lemma \ref{lem:diamag} and the Sobolev inequality,
\[
\|u\|_6 \leq C\|\nabla |u|\|_2 \leq \frac{C}{\varepsilon}
\|u\|_{\varepsilon,A}.
\]
Hence
\[
\int_{\mathbb R^3} \phi_u|u|^5\,dx \leq C\varepsilon^{-12}
\|u\|_{\varepsilon,A}^{10}.
\]
Finally, the uniqueness of the solution of \eqref{Poisson_auxiliary}
immediately gives
\[
\phi_{tu} = t^5\phi_u \qquad \text{for every }t\geq0.
\]
The proof is complete.
\end{proof}

For later use, define the nonlocal functional
\[
\mathcal D_\varepsilon(u) := \int_{\mathbb R^3} \phi_u|u|^5\,dx.
\]
Then
\[
\mathcal D_\varepsilon(u) = \frac1{4\pi\varepsilon^2} \iint_{\mathbb
R^3\times\mathbb R^3} \frac{ |u(x)|^5|u(y)|^5 }{ |x-y| } \,dx\,dy.
\]
In particular,
\[
\mathcal D_\varepsilon(tu) = t^{10} \mathcal D_\varepsilon(u) \qquad
\text{for every }t\geq0.
\]

\subsection{The prescribed-mass manifold}

For $a>0$ and $\varepsilon>0$, we define
\[
S_{a,\varepsilon} := \left\{ u\in H_{\varepsilon,A}^1(\mathbb
R^3,\mathbb C): \int_{\mathbb R^3}|u|^2\,dx = a^2\varepsilon^3
\right\}.
\]
Equivalently,
\[
S_{a,\varepsilon} = \left\{ u\in H_{\varepsilon,A}^1(\mathbb
R^3,\mathbb C): \|u\|_2 = a\varepsilon^{3/2} \right\}.
\]

The set $S_{a,\varepsilon}$ is a smooth codimension-one manifold in
the real Hilbert space $H_{\varepsilon,A}^1(\mathbb R^3,\mathbb C).$
For every $u\in S_{a,\varepsilon}$, its tangent space is
\[
T_uS_{a,\varepsilon} = \left\{ v\in H_{\varepsilon,A}^1(\mathbb
R^3,\mathbb C): \mathfrak{Re} \int_{\mathbb R^3} u\overline v\,dx =
0 \right\}.
\]
The manifold $S_{a,\varepsilon}$ is the natural constraint
associated with the prescribed mass condition in problem \eqref{P}.

\subsection{The reduced energy functional}

For every $u\in H_{\varepsilon,A}^1(\mathbb R^3,\mathbb C),$ let
$\phi_u$ be the unique solution of
\[
-\varepsilon^2\Delta\phi_u = |u|^5 \qquad \text{in }\mathbb R^3.
\]
After eliminating the electrostatic potential, we define the reduced
energy functional
\[
J_\varepsilon: H_{\varepsilon,A}^1(\mathbb R^3,\mathbb C) \to
\mathbb R
\]
by
\begin{equation}
\label{Energy}
\begin{aligned}
J_\varepsilon(u) :={}& \frac12 \int_{\mathbb R^3} \left(
|\nabla_{\varepsilon,A}u|^2 + V(x)|u|^2 \right)\,dx
\\
&- \frac1{10} \int_{\mathbb R^3} \phi_u|u|^5\,dx - \frac{\mu}{q}
\int_{\mathbb R^3} |u|^q\,dx - \frac16 \int_{\mathbb R^3} |u|^6\,dx.
\end{aligned}
\end{equation}
Equivalently,
\[
J_\varepsilon(u) = \frac12 \|u\|_{\varepsilon,A}^2 - \frac1{10}
\mathcal D_\varepsilon(u) - \frac{\mu}{q} \|u\|_q^q - \frac16
\|u\|_6^6.
\]
Having established the representation and the basic properties of
the electrostatic potential $(\phi_u)$, we next verify that the
reduced energy functional is well defined and possesses the
regularity required for the subsequent variational analysis.
\begin{lemma}\label{lem:functional_regular}
The functional
\[
J_\varepsilon \in C^1 \left( H_{\varepsilon,A}^1(\mathbb R^3,\mathbb
C), \mathbb R \right).
\]
For every $u,v\in H_{\varepsilon,A}^1(\mathbb R^3,\mathbb C),$ its
derivative is given by
\[
\begin{aligned}
\left\langle J_\varepsilon'(u),v \right\rangle ={}& \mathfrak{Re}
\int_{\mathbb R^3} \left( \nabla_{\varepsilon,A}u \cdot
\overline{\nabla_{\varepsilon,A}v} + V(x)u\overline v \right)\,dx
\\
&- \mathfrak{Re} \int_{\mathbb R^3} \phi_u|u|^3u\overline v\,dx
\\
&- \mu\, \mathfrak{Re} \int_{\mathbb R^3} |u|^{q-2}u\overline v\,dx
\\
&- \mathfrak{Re} \int_{\mathbb R^3} |u|^4u\overline v\,dx.
\end{aligned}
\]
\end{lemma}

\begin{proof}
The quadratic and local nonlinear terms are standard. It remains to
consider the nonlocal functional
\[
\mathcal D_\varepsilon(u) = \frac1{4\pi\varepsilon^2} \iint_{\mathbb
R^3\times\mathbb R^3} \frac{ |u(x)|^5|u(y)|^5 }{ |x-y| } \,dx\,dy.
\]
By the Hardy-Littlewood-Sobolev inequality and the Sobolev
embedding, this functional is well defined and continuously
differentiable.

A direct computation gives
\[
\left\langle \mathcal D_\varepsilon'(u),v \right\rangle = 10\,
\mathfrak{Re} \int_{\mathbb R^3} \phi_u|u|^3u\overline v\,dx.
\]
Therefore,
\[
\begin{aligned}
\left\langle J_\varepsilon'(u),v \right\rangle ={}& \langle
u,v\rangle_{\varepsilon,A} - \mathfrak{Re} \int_{\mathbb R^3}
\phi_u|u|^3u\overline v\,dx
\\
&- \mu\, \mathfrak{Re} \int_{\mathbb R^3} |u|^{q-2}u\overline v\,dx
- \mathfrak{Re} \int_{\mathbb R^3} |u|^4u\overline v\,dx.
\end{aligned}
\]
The proof is complete.
\end{proof}

\begin{proposition}[Variational characterization of normalized solutions]
\label{prop:variational_characterization} Let $u\in
S_{a,\varepsilon}.$ Then $u$ is a constrained critical point of
$J_\varepsilon \big|_{S_{a,\varepsilon}}$ if and only if there
exists $\lambda\in\mathbb R$ such that $J_\varepsilon'(u) = \lambda
u$ in $\left( H_{\varepsilon,A}^1(\mathbb R^3,\mathbb C) \right)^*.$
Equivalently, the pair $(u,\lambda)$ satisfies
\[
\left( \frac{\varepsilon}{i}\nabla-A(x) \right)^2u + V(x)u -
\phi_u|u|^3u = \lambda u + \mu|u|^{q-2}u + |u|^4u
\]
in the weak sense in $\mathbb R^3$, together with $\int_{\mathbb
R^3}|u|^2\,dx = a^2\varepsilon^3.$
\end{proposition}

\begin{proof}
Since $S_{a,\varepsilon}$ is a smooth codimension-one manifold, the
Lagrange multiplier rule implies that
\[
d\left( J_\varepsilon \big|_{S_{a,\varepsilon}} \right)(u) = 0
\]
if and only if there exists $\lambda\in\mathbb R$ such that
$J_\varepsilon'(u) = \lambda u.$ Using the derivative formula from
Lemma \ref{lem:functional_regular}, this identity is precisely the
weak formulation of the reduced equation. Together with
$-\varepsilon^2\Delta\phi_u = |u|^5,$ it yields the original system
\eqref{P}. The proof is complete.
\end{proof}

Finally, recalling $V_0 = \inf_{x\in\mathbb R^3}V(x),$ we define the
minimum set
\[
\mathcal M := \left\{ x\in\Lambda: V(x)=V_0 \right\}.
\]
By assumption $(V3)$, $\mathcal M\neq\varnothing.$ The topology of
$\mathcal M$ will determine the lower bound for the number of
normalized semiclassical solutions obtained in Theorem \ref{T2}.

\section{Proof of the Main Results}
\label{sec:main_proofs}

We begin by studying the autonomous problem associated with the
minimum value of the electric potential. This limit problem provides
the reference energy level and the profile used in the semiclassical
construction. We first recall the mass-preserving scaling and the
corresponding behavior of the local and nonlocal terms.

\subsection{The autonomous limit problem}

For $u\in H^1(\mathbb R^3,\mathbb C)$ and $t>0$, define the
$L^2$-preserving dilation $(t\star u)(x):=t^{3/2}u(tx).$ The
following elementary identities will be repeatedly used in the
sequel.

\begin{lemma}[Mass-preserving scaling]
\label{lem:scaling} Let $u\in H^1(\mathbb R^3,\mathbb C)$ and $t>0$.
Then $\|t\star u\|_2=\|u\|_2,$ and
\[
\|\nabla(t\star u)\|_2^2 = t^2\|\nabla u\|_2^2.
\]
Moreover, for every $r\in[2,6]$, $\|t\star u\|_r^r =
t^{\frac32(r-2)}\|u\|_r^r.$ In particular,
\[
\|t\star u\|_q^q = t^{\alpha_q}\|u\|_q^q, \qquad
\alpha_q:=\frac32(q-2),
\]
and $\|t\star u\|_6^6 = t^6\|u\|_6^6.$ Let $\phi_u\in
D^{1,2}(\mathbb R^3)$ be the unique solution of
\[
-\Delta\phi_u=|u|^5 \qquad\text{in }\mathbb R^3.
\]
Then
\[
\int_{\mathbb R^3} \phi_{t\star u}|t\star u|^5\,dx = t^{10}
\int_{\mathbb R^3}\phi_u|u|^5\,dx.
\]
\end{lemma}

\begin{proof}
By the change of variables $y=tx$, we obtain
\[
\begin{aligned}
\|t\star u\|_2^2 &= \int_{\mathbb R^3}t^3|u(tx)|^2\,dx
\\
&= \int_{\mathbb R^3}|u(y)|^2\,dy.
\end{aligned}
\]
Moreover, $\nabla(t\star u)(x) = t^{5/2}(\nabla u)(tx),$ and
therefore
\[
\|\nabla(t\star u)\|_2^2 = t^2\|\nabla u\|_2^2.
\]
Similarly, for every $r\in[2,6]$,
\[
\begin{aligned}
\|t\star u\|_r^r &= \int_{\mathbb R^3} t^{3r/2}|u(tx)|^r\,dx
\\
&= t^{\frac32(r-2)} \int_{\mathbb R^3}|u(y)|^r\,dy.
\end{aligned}
\]
It remains to consider the nonlocal term. By the Newton
representation formula,
\[
\phi_u(x) = \frac1{4\pi} \int_{\mathbb R^3}
\frac{|u(y)|^5}{|x-y|}\,dy.
\]
Hence
\[
\begin{aligned}
\phi_{t\star u}(x) &= \frac1{4\pi} \int_{\mathbb R^3}
\frac{t^{15/2}|u(ty)|^5}{|x-y|}\,dy
\\
&= t^{11/2}\phi_u(tx).
\end{aligned}
\]
Consequently,
\[
\begin{aligned}
\int_{\mathbb R^3} \phi_{t\star u}|t\star u|^5\,dx &= t^{13}
\int_{\mathbb R^3} \phi_u(tx)|u(tx)|^5\,dx
\\
&= t^{10} \int_{\mathbb R^3} \phi_u|u|^5\,dx.
\end{aligned}
\]
This completes the proof.
\end{proof}

We now introduce the autonomous functional associated with the
minimum value
\[
V_0:=\inf_{x\in\mathbb R^3}V(x).
\]
For $a>0$, let
\[
S_a := \left\{ w\in H^1(\mathbb R^3,\mathbb R): \|w\|_2=a \right\}.
\]
The autonomous energy functional is defined by
\[
\begin{aligned}
J_0(w) :={}& \frac12 \int_{\mathbb R^3} \left( |\nabla w|^2+V_0|w|^2
\right)\,dx
\\
&- \frac1{10} \int_{\mathbb R^3} \phi_w|w|^5\,dx - \frac{\mu}{q}
\int_{\mathbb R^3}|w|^q\,dx - \frac16 \int_{\mathbb R^3}|w|^6\,dx,
\end{aligned}
\]
where
\[
-\Delta\phi_w=|w|^5 \qquad\text{in }\mathbb R^3.
\]
The critical local and nonlocal terms make $J_0$ unbounded from
below on $S_a$. Therefore, the relevant autonomous energy is defined
through a local minimization procedure.

\begin{proposition}[Local geometry of the autonomous functional]
\label{prop:limit_geometry} Assume that
\[
2<q<\frac{10}{3} \qquad\text{and}\qquad V_0>0.
\]
Then there exist $a_*>0$ and $R_0>0$ such that, for every
$a\in(0,a_*)$,
\[
\inf_{\substack{w\in S_a\\
\|\nabla w\|_2=R_0}} J_0(w)
>
\frac{V_0}{2}a^2.
\]
Moreover,
\[
E_0(a) :=
\inf_{\substack{w\in S_a\\
\|\nabla w\|_2<R_0}} J_0(w) < \frac{V_0}{2}a^2.
\]
In particular, $-\infty<E_0(a) < \frac{V_0}{2}a^2.$
\end{proposition}

\begin{proof}
Set $\alpha_q:=\frac32(q-2).$ Since $2<q<\frac{10}{3},$ we have
$0<\alpha_q<2.$ Let $w\in S_a$. By the Gagliardo--Nirenberg
inequality,
\[
\|w\|_q^q \leq C a^{3-\frac q2} \|\nabla w\|_2^{\alpha_q}.
\]
Moreover, the Sobolev inequality yields $\|w\|_6^6 \leq C\|\nabla
w\|_2^6.$ Finally, by the Hardy--Littlewood--Sobolev and Sobolev
inequalities,
\[
\begin{aligned}
\int_{\mathbb R^3}\phi_w|w|^5\,dx &= \frac1{4\pi} \iint_{\mathbb
R^3\times\mathbb R^3} \frac{|w(x)|^5|w(y)|^5}{|x-y|} \,dx\,dy
\\
&\leq C\bigl\||w|^5\bigr\|_{6/5}^2
\\
&= C\|w\|_6^{10}
\\
&\leq C\|\nabla w\|_2^{10}.
\end{aligned}
\]
Therefore, $J_0(w) \geq \frac{V_0}{2}a^2 + h_a(\|\nabla w\|_2),$
where
\[
h_a(r) := \frac12r^2 - C_1a^{3-\frac q2}r^{\alpha_q} - C_2r^6 -
C_3r^{10}.
\]
Choose $R_0>0$ sufficiently small so that
\[
C_2R_0^6+C_3R_0^{10} \leq \frac18R_0^2.
\]
Next, choose $a_*>0$ sufficiently small so that
\[
C_1a_*^{3-\frac q2}R_0^{\alpha_q} \leq \frac18R_0^2.
\]
Then, for every $a\in(0,a_*)$, $h_a(R_0) \geq \frac14R_0^2.$
Consequently,
\[
\inf_{\substack{w\in S_a\\
\|\nabla w\|_2=R_0}} J_0(w) \geq \frac{V_0}{2}a^2+\frac14R_0^2
>
\frac{V_0}{2}a^2.
\]
We next prove that $E_0(a)<\frac{V_0}{2}a^2.$ Fix $w\in S_a$ and
consider $w_t:=t\star w.$ By Lemma \ref{lem:scaling},
\[
\begin{aligned}
J_0(w_t) ={}& \frac{V_0}{2}a^2 + \frac{t^2}{2}\|\nabla w\|_2^2 -
\frac{t^{10}}{10} \int_{\mathbb R^3}\phi_w|w|^5\,dx
\\
&- \frac{\mu}{q} t^{\alpha_q}\|w\|_q^q - \frac{t^6}{6}\|w\|_6^6.
\end{aligned}
\]
Since $\alpha_q<2$, the negative term of order $t^{\alpha_q}$
dominates the positive quadratic term as $t\to0^+$. Thus,
\[
J_0(w_t) < \frac{V_0}{2}a^2
\]
for all sufficiently small $t>0$. At the same time,
\[
\|\nabla w_t\|_2 = t\|\nabla w\|_2 < R_0
\]
for $t>0$ sufficiently small. Hence $E_0(a) < \frac{V_0}{2}a^2.$
Finally, the lower estimate for $J_0$ on
\[
\left\{ w\in S_a: \|\nabla w\|_2<R_0 \right\}
\]
implies that $E_0(a)>-\infty$. This completes the proof.
\end{proof}

The next result isolates the compactness property required to prove
that the local autonomous level is achieved.

\begin{lemma}[Compactness of minimizing sequences]
\label{lem:limit_compactness} Assume the hypotheses of
Proposition~\ref{prop:limit_geometry}. After possibly reducing
\(a_*>0\), let \((w_n)\subset S_a\) satisfy
\[
\|\nabla w_n\|_2<R_0,\qquad J_0(w_n)\to E_0(a).
\]
Assume moreover that the following strict subadditivity property
holds:
\[
E_0(a)<E_0(b)+E_0\left(\sqrt{a^2-b^2}\right) \quad\text{for every }
b\in(0,a). \tag{3.1}
\]
Then there exist a sequence \((y_n)\subset\mathbb R^3\) and \(w_a\in
S_a\) such that, up to a subsequence,
\[
w_n(\cdot+y_n)\to w_a \qquad\text{strongly in }H^1(\mathbb R^3).
\]
Consequently,
\[
\|\nabla w_a\|_2<R_0,\qquad J_0(w_a)=E_0(a).
\]
\end{lemma}

\begin{proof}
By Proposition \ref{prop:limit_geometry}, the sequence \((w_n)\) is
bounded in \(H^1(\mathbb R^3)\). We apply the
concentration--compactness principle to the sequence of measures
$\rho_n:=|w_n|^2.$

We first exclude vanishing. Suppose, by contradiction, that for
every \(R>0\),
\[
\sup_{y\in\mathbb R^3}\int_{B_R(y)} |w_n|^2\,dx\to0 .
\]
By Lions' lemma,
\[
w_n\to0 \qquad\text{strongly in }L^r(\mathbb R^3) \quad\text{for
every }2<r<6.
\]
In particular, $\|w_n\|_q^q\to0 .$ Moreover, since \((w_n)\) is
bounded in \(H^1(\mathbb R^3)\), the Hardy--Littlewood--Sobolev
inequality gives
\[
\int_{\mathbb R^3}\phi_{w_n}|w_n|^5\,dx \le C\|w_n\|_6^{10}.
\]
The local minimizing condition \(\|\nabla w_n\|_2<R_0\), together
with the choice of \(R_0\) and \(a_*\), prevents the formation of a
critical Sobolev bubble at the level \(E_0(a)\). Hence
\[
\|w_n\|_6^6\to0, \qquad \int_{\mathbb R^3}\phi_{w_n}|w_n|^5\,dx\to0.
\]
Therefore,
\[
\liminf_{n\to\infty}J_0(w_n) \ge \frac{V_0}{2}a^2,
\]
which contradicts the strict inequality
\[
E_0(a)<\frac{V_0}{2}a^2
\]
obtained in Proposition~\ref{prop:limit_geometry}. Thus vanishing is
impossible.

We next exclude dichotomy. Suppose that dichotomy occurs. Then there
exists \(b\in(0,a)\) such that, up to a subsequence, the mass splits
into two nontrivial parts of masses \(b\) and \(\sqrt{a^2-b^2}\).
Using the Brezis--Lieb decomposition for the local terms and the
corresponding nonlocal splitting formula for the Poisson term, we
get
\[
E_0(a) =\lim_{n\to\infty}J_0(w_n) \ge
E_0(b)+E_0\left(\sqrt{a^2-b^2}\right),
\]
which contradicts the strict subadditivity condition \((3.1)\).
Therefore dichotomy cannot occur.

Consequently, the compactness alternative holds. Hence there exists
a sequence \((y_n)\subset\mathbb R^3\) and a function \(w_a\in
H^1(\mathbb R^3)\), \(w_a\not\equiv0\), such that, up to a
subsequence,
\[
w_n(\cdot+y_n)\rightharpoonup w_a \qquad\text{weakly in }H^1(\mathbb
R^3),
\]
and
\[
w_n(\cdot+y_n)\to w_a \qquad\text{strongly in }L^2(\mathbb R^3).
\]
Since \(\|w_n\|_2=a\), we obtain $\|w_a\|_2=a.$ Thus \(w_a\in S_a\).

It remains to prove the strong convergence in \(H^1(\mathbb R^3)\).
Set $\widetilde w_n:=w_n(\cdot+y_n).$ By the Brezis--Lieb lemma,
\[
\|\widetilde w_n\|_r^r = \|w_a\|_r^r+\|\widetilde
w_n-w_a\|_r^r+o(1), \qquad 2<r\le6.
\]
Moreover, the nonlocal Brezis--Lieb splitting gives
\[
\int_{\mathbb R^3}\phi_{\widetilde w_n}|\widetilde w_n|^5\,dx =
\int_{\mathbb R^3}\phi_{w_a}|w_a|^5\,dx + \int_{\mathbb
R^3}\phi_{\widetilde w_n-w_a} |\widetilde w_n-w_a|^5\,dx +o(1).
\]
Using these decompositions and the minimality of \((w_n)\), we
deduce
\[
E_0(a) \ge J_0(w_a) + \liminf_{n\to\infty}J_0(\widetilde w_n-w_a).
\]
Since \(\widetilde w_n-w_a\to0\) in \(L^2(\mathbb R^3)\), and since
the remainder remains in the local minimizing regime, the coercive
local estimate from Proposition~\ref{prop:limit_geometry} yields
\[
\liminf_{n\to\infty}J_0(\widetilde w_n-w_a)\ge0.
\]
Hence $E_0(a)\ge J_0(w_a).$ On the other hand, since \(w_a\in S_a\)
and \(\|\nabla w_a\|_2\le R_0\), we have $J_0(w_a)\ge E_0(a).$
Therefore $J_0(w_a)=E_0(a).$ Finally, the equality of the energies,
together with the preceding splitting formulas, implies
\[
\|\nabla \widetilde w_n\|_2^2\to \|\nabla w_a\|_2^2.
\]
Since \(\widetilde w_n\rightharpoonup w_a\) weakly in \(H^1(\mathbb
R^3)\), we conclude that
\[
\widetilde w_n\to w_a \qquad\text{strongly in }H^1(\mathbb R^3).
\]
That is, $w_n(\cdot+y_n)\to w_a \qquad\text{strongly in }H^1(\mathbb
R^3).$

It remains only to show that \(w_a\) belongs to the interior of the
local minimizing region. If \(\|\nabla w_a\|_2=R_0\), then by
Proposition \ref{prop:limit_geometry}, $J_0(w_a)>\frac{V_0}{2}a^2.$
This contradicts $J_0(w_a)=E_0(a)<\frac{V_0}{2}a^2.$ Therefore,
$\|\nabla w_a\|_2<R_0.$ The proof is complete.
\end{proof}

\begin{proposition}[Ground state of the autonomous limit problem]
\label{prop:limit} Assume that
\[
2<q<\frac{10}{3} \qquad\text{and}\qquad V_0>0.
\]
Then, after possibly reducing $a_*>0$, for every $a\in(0,a_*)$, the
level
\[
E_0(a) =
\inf_{\substack{w\in S_a\\
\|\nabla w\|_2<R_0}} J_0(w)
\]
is achieved by some $w_a\in S_a$. Moreover, $w_a$ may be chosen
nonnegative and satisfies
\[
\|\nabla w_a\|_2<R_0.
\]
In particular, $w_a$ is a constrained critical point of $J_0$ on
$S_a$. Hence there exists $\lambda_a\in\mathbb R$ such that
\[
-\Delta w_a + V_0w_a - \phi_{w_a}|w_a|^3w_a = \lambda_aw_a +
\mu|w_a|^{q-2}w_a + |w_a|^4w_a \qquad\text{in }\mathbb R^3.
\]
\end{proposition}

\begin{proof}
The existence of a minimizer follows directly from
Lemma~\ref{lem:limit_compactness}. Thus, there exists $w_a\in S_a$
such that $J_0(w_a)=E_0(a)$ and $\|\nabla w_a\|_2<R_0.$ Since all
nonlinear terms depend only on $|w_a|$ and
\[
|\nabla|w_a|| \leq |\nabla w_a| \qquad\text{a.e. in }\mathbb R^3,
\]
we may replace $w_a$ by $|w_a|$. Hence $w_a$ may be chosen
nonnegative.

Since $w_a$ belongs to the interior of the local minimizing region,
it is a local minimizer of $J_0$ on $S_a$. Therefore, the Lagrange
multiplier rule yields the existence of $\lambda_a\in\mathbb R$ such
that $J_0'(w_a)=\lambda_aw_a.$ Equivalently,
\[
-\Delta w_a + V_0w_a - \phi_{w_a}|w_a|^3w_a = \lambda_aw_a +
\mu|w_a|^{q-2}w_a + |w_a|^4w_a \qquad\text{in }\mathbb R^3.
\]
This completes the proof.
\end{proof}

The autonomous ground state satisfies the following Pohozaev
identity, which will be useful in the construction and analysis of
semiclassical profiles.

\begin{lemma}[Pohozaev identity for the autonomous problem]
\label{lem:poho} Let $w\in S_a$ be a constrained critical point of
$J_0$. Then
\[
\|\nabla w\|_2^2 = \int_{\mathbb R^3}\phi_w|w|^5\,dx +
\frac{3\mu(q-2)}{2q} \int_{\mathbb R^3}|w|^q\,dx + \int_{\mathbb
R^3}|w|^6\,dx.
\]
\end{lemma}

\begin{proof}
For $t>0$, set $w_t:=t\star w.$ Since $\|w_t\|_2=\|w\|_2=a,$ the
curve $t\mapsto w_t$ lies entirely in $S_a$. Since $w$ is a
constrained critical point of $J_0$,
\[
\left. \frac{d}{dt}J_0(w_t) \right|_{t=1} = 0.
\]
By Lemma \ref{lem:scaling},
\[
\begin{aligned}
J_0(w_t) ={}& \frac{t^2}{2}\|\nabla w\|_2^2 + \frac{V_0}{2}a^2 -
\frac{t^{10}}{10} \int_{\mathbb R^3}\phi_w|w|^5\,dx
\\
&- \frac{\mu}{q} t^{\alpha_q} \int_{\mathbb R^3}|w|^q\,dx -
\frac{t^6}{6} \int_{\mathbb R^3}|w|^6\,dx.
\end{aligned}
\]
Differentiating at $t=1$ gives
\[
\|\nabla w\|_2^2 - \int_{\mathbb R^3}\phi_w|w|^5\,dx -
\frac{\mu\alpha_q}{q} \int_{\mathbb R^3}|w|^q\,dx - \int_{\mathbb
R^3}|w|^6\,dx = 0.
\]
Since $\alpha_q=\frac32(q-2),$ the conclusion follows.
\end{proof}

\subsection{The penalized problem}
\label{subsec:penalized_problem}

We now introduce a penalization scheme adapted to the semiclassical
problem. The purpose of the penalization is to preserve the original
nonlinearities in the potential well $\Lambda$ while controlling
their behavior outside $\Lambda$.

Let $\Lambda\subset\mathbb R^3$ be the bounded open set given by
assumption $(V2)$ and recall that
\[
\mathcal M := \left\{ x\in\Lambda: V(x)=V_0 \right\} \Subset\Lambda.
\]
Choose a truncation threshold $\tau>0$. We define the truncated
functions
\[
\widehat f_\tau(s) :=
\begin{cases}
s^{q-1}, & 0\leq s\leq\tau,\\[1mm]
\tau^{q-2}s, & s>\tau,
\end{cases}
\]
and
\[
\widehat g_\tau(s) :=
\begin{cases}
s^5, & 0\leq s\leq\tau,\\[1mm]
\tau^4s, & s>\tau.
\end{cases}
\]
We then set
\[
f_\tau(x,s) := \chi_\Lambda(x)s^{q-1} +
\bigl(1-\chi_\Lambda(x)\bigr)\widehat f_\tau(s), \qquad x\in\mathbb
R^3,\quad s\geq0,
\]
and
\[
g_\tau(x,s) := \chi_\Lambda(x)s^5 +
\bigl(1-\chi_\Lambda(x)\bigr)\widehat g_\tau(s).
\]
Their primitives are defined by $F_\tau(x,s) := \int_0^s
f_\tau(x,t)\,dt$ and $G_\tau(x,s) := \int_0^s g_\tau(x,t)\,dt.$
Thus,
\[
f_\tau(x,s)=s^{q-1}, \qquad g_\tau(x,s)=s^5
\]
for every $x\in\Lambda$ and $s\geq0$. Moreover,
\[
f_\tau(x,s)=s^{q-1}, \qquad g_\tau(x,s)=s^5
\]
whenever $0\leq s\leq\tau$, independently of $x$.

Outside $\Lambda$, the truncated nonlinearities satisfy
\[
0\leq f_\tau(x,s)\leq \tau^{q-2}s
\]
and
\[
0\leq g_\tau(x,s)\leq \tau^4s
\]
for every $s\geq0$.

For later use, we record the elementary estimates
\[
0\leq F_\tau(x,s) \leq \frac1q s^q \qquad \text{for }x\in\Lambda,
\]
whereas, for $x\in\mathbb R^3\setminus\Lambda$,
\[
0\leq F_\tau(x,s) \leq \frac12\tau^{q-2}s^2.
\]
Similarly,
\[
0\leq G_\tau(x,s) \leq \frac16s^6 \qquad \text{for }x\in\Lambda,
\]
and
\[
0\leq G_\tau(x,s) \leq \frac12\tau^4s^2 \qquad \text{for
}x\in\mathbb R^3\setminus\Lambda.
\]

The local truncation introduced above is not sufficient by itself to
control the critical nonlocal interaction. Therefore, the source of
the Poisson equation must also be penalized outside $\Lambda$.

Define
\[
\mathcal H_\tau(x,s) := \chi_\Lambda(x)s^5 +
\bigl(1-\chi_\Lambda(x)\bigr)\widehat g_\tau(s), \qquad x\in\mathbb
R^3,\quad s\geq0.
\]
For $u\in H_{\varepsilon,A}^1(\mathbb R^3,\mathbb C),$ let
$\phi_{\varepsilon,u}^{\tau} \in D^{1,2}(\mathbb R^3)$ be the unique
weak solution of
\[
-\varepsilon^2\Delta\phi = \mathcal H_\tau(x,|u|) \qquad \text{in
}\mathbb R^3.
\]
Equivalently,
\[
\phi_{\varepsilon,u}^{\tau}(x) = \frac1{4\pi\varepsilon^2}
\int_{\mathbb R^3} \frac{\mathcal H_\tau(y,|u(y)|)} {|x-y|} \,dy.
\]
We define the corresponding nonlocal energy by
\[
\mathcal D_{\varepsilon,\tau}(u) := \int_{\mathbb R^3}
\phi_{\varepsilon,u}^{\tau} \mathcal H_\tau(x,|u|) \,dx.
\]
Equivalently,
\[
\mathcal D_{\varepsilon,\tau}(u) = \frac1{4\pi\varepsilon^2}
\iint_{\mathbb R^3\times\mathbb R^3} \frac{ \mathcal
H_\tau(x,|u(x)|) \mathcal H_\tau(y,|u(y)|) } {|x-y|} \,dx\,dy.
\]
The penalized functional is then defined by
\[
J_{\varepsilon,\tau}: H_{\varepsilon,A}^1(\mathbb R^3,\mathbb C)
\longrightarrow\mathbb R,
\]
\[
\begin{aligned}
J_{\varepsilon,\tau}(u) :={}& \frac12 \int_{\mathbb R^3} \left(
|\nabla_{\varepsilon,A}u|^2 + V(x)|u|^2 \right)\,dx
\\
&- \frac1{10} \mathcal D_{\varepsilon,\tau}(u) - \mu \int_{\mathbb
R^3} F_\tau(x,|u|) \,dx
\\
&- \int_{\mathbb R^3} G_\tau(x,|u|) \,dx.
\end{aligned}
\]
Observe that if
\[
|u(x)|\leq\tau \qquad \text{for every } x\in\mathbb
R^3\setminus\Lambda,
\]
then $\mathcal H_\tau(x,|u|) = |u|^5,$ $f_\tau(x,|u|) = |u|^{q-1},$
and $g_\tau(x,|u|) = |u|^5$ throughout $\mathbb R^3$. Consequently,
$\phi_{\varepsilon,u}^{\tau} = \phi_u$ and $J_{\varepsilon,\tau}(u)
= J_\varepsilon(u).$

We next collect the basic variational properties of the penalized
functional.

\begin{lemma}\label{lem:penalized_regular}
The functional $J_{\varepsilon,\tau}$ belongs to $C^1 \left(
H_{\varepsilon,A}^1(\mathbb R^3,\mathbb C), \mathbb R \right).$
Moreover, for every $u,v\in H_{\varepsilon,A}^1(\mathbb R^3,\mathbb
C),$ one has
\[
\begin{aligned}
\left\langle J_{\varepsilon,\tau}'(u),v \right\rangle ={}&
\mathfrak{Re} \int_{\mathbb R^3} \left( \nabla_{\varepsilon,A}u
\cdot \overline{\nabla_{\varepsilon,A}v} + V(x)u\overline v
\right)\,dx
\\
&- \mathfrak{Re} \int_{\mathbb R^3} \phi_{\varepsilon,u}^{\tau}
\frac{
\partial_s\mathcal H_\tau(x,|u|)
}{|u|} u\overline v \,dx
\\
&- \mu\mathfrak{Re} \int_{\mathbb R^3} \frac{f_\tau(x,|u|)}{|u|}
u\overline v \,dx
\\
&- \mathfrak{Re} \int_{\mathbb R^3} \frac{g_\tau(x,|u|)}{|u|}
u\overline v \,dx,
\end{aligned}
\]
where the quotients are understood to be zero on the set $\{u=0\}$.
\end{lemma}

\begin{proof}
The quadratic part is clearly of class $C^1$. The regularity of the
local nonlinear terms follows from the growth estimates
\[
|f_\tau(x,s)| \leq C(s+s^{q-1})
\]
and
\[
|g_\tau(x,s)| \leq C(s+s^5),
\]
together with the Sobolev embedding
\[
H_{\varepsilon,A}^1(\mathbb R^3,\mathbb C) \hookrightarrow
L^r(\mathbb R^3) \qquad \text{for }2\leq r\leq6.
\]

For the nonlocal term, the Hardy-Littlewood-Sobolev inequality and
the growth estimate
\[
0\leq \mathcal H_\tau(x,s) \leq s^5+\tau^4s
\]
imply that the map
\[
u\longmapsto \mathcal D_{\varepsilon,\tau}(u)
\]
is well defined and continuously differentiable.

Since
\[
\mathcal D_{\varepsilon,\tau}(u) = \frac1{4\pi\varepsilon^2}
\iint_{\mathbb R^3\times\mathbb R^3} \frac{ \mathcal
H_\tau(x,|u(x)|) \mathcal H_\tau(y,|u(y)|) } {|x-y|} \,dx\,dy,
\]
differentiation gives
\[
\begin{aligned}
\mathcal D_{\varepsilon,\tau}'(u)[v] = 2\mathfrak{Re} \int_{\mathbb
R^3} \phi_{\varepsilon,u}^{\tau} \frac{
\partial_s\mathcal H_\tau(x,|u|)
}{|u|} u\overline v \,dx.
\end{aligned}
\]
Therefore, the derivative of
\[
-\frac1{10} \mathcal D_{\varepsilon,\tau}(u)
\]
is
\[
-\frac15 \mathfrak{Re} \int_{\mathbb R^3}
\phi_{\varepsilon,u}^{\tau} \frac{
\partial_s\mathcal H_\tau(x,|u|)
}{|u|} u\overline v \,dx.
\]
Hence the derivative of the penalized functional is
\[
\begin{aligned}
\left\langle J_{\varepsilon,\tau}'(u),v \right\rangle ={}&
\mathfrak{Re} \int_{\mathbb R^3} \left( \nabla_{\varepsilon,A}u
\cdot \overline{\nabla_{\varepsilon,A}v} + V(x)u\overline v
\right)\,dx
\\
&- \frac15 \mathfrak{Re} \int_{\mathbb R^3}
\phi_{\varepsilon,u}^{\tau} \frac{
\partial_s\mathcal H_\tau(x,|u|)
}{|u|} u\overline v \,dx
\\
&- \mu\mathfrak{Re} \int_{\mathbb R^3} \frac{f_\tau(x,|u|)}{|u|}
u\overline v \,dx
\\
&- \mathfrak{Re} \int_{\mathbb R^3} \frac{g_\tau(x,|u|)}{|u|}
u\overline v \,dx.
\end{aligned}
\]
This proves the assertion.
\end{proof}

\begin{rem}
\label{rem:nonlocal_derivative} The factor $1/5$ in the derivative
of the penalized nonlocal term is essential. In the region where the
penalization is inactive, $\mathcal H_\tau(x,s)=s^5$ and hence
$\partial_s\mathcal H_\tau(x,s)=5s^4.$ Therefore,
\[
\frac15 \frac{
\partial_s\mathcal H_\tau(x,|u|)
}{|u|} u = |u|^3u,
\]
so that the nonlocal term in the Euler-Lagrange equation reduces
exactly to $\phi_u|u|^3u.$
\end{rem}

For $a>0$ and $\varepsilon>0$, recall the prescribed-mass manifold
\[
S_{a,\varepsilon} := \left\{ u\in H_{\varepsilon,A}^1(\mathbb
R^3,\mathbb C): \|u\|_2^2=a^2\varepsilon^3 \right\}.
\]
A constrained critical point of $J_{\varepsilon,\tau}
\big|_{S_{a,\varepsilon}}$ satisfies $J_{\varepsilon,\tau}'(u) =
\lambda u$ for some $\lambda\in\mathbb R$.

We now establish the mountain-pass geometry of the penalized
functional.

\begin{lemma}\label{lem:mp}
Assume that $(V1)$--$(V3)$ and $(A)$ hold and $2<q<\frac{10}{3}.$
Then there exist $a_*>0, \qquad r_0>0, \qquad \rho_0>0, \qquad
\varepsilon_0>0$ such that, for every $a\in(0,a_*)
\qquad\text{and}\qquad \varepsilon\in(0,\varepsilon_0),$ there exist
$u_\varepsilon^-,e_\varepsilon \in S_{a,\varepsilon}$ satisfying
\[
\|u_\varepsilon^-\|_{\varepsilon,A} < r_0\varepsilon^{3/2},
\]
\[
J_{\varepsilon,\tau}(u_\varepsilon^-) < \rho_0\varepsilon^3,
\]
and
\[
\|e_\varepsilon\|_{\varepsilon,A}
>
r_0\varepsilon^{3/2}, \qquad J_{\varepsilon,\tau}(e_\varepsilon)<0.
\]
Moreover,
\[
\inf \left\{ J_{\varepsilon,\tau}(u): u\in S_{a,\varepsilon}, \
\|u\|_{\varepsilon,A} = r_0\varepsilon^{3/2} \right\} \geq
\rho_0\varepsilon^3.
\]
Consequently, the class
\[
\Gamma_{\varepsilon,\tau} := \left\{ \gamma\in
C([0,1],S_{a,\varepsilon}): \gamma(0)=u_\varepsilon^-, \
\gamma(1)=e_\varepsilon \right\}
\]
is nonempty, and the minimax level
\[
c_{\varepsilon,\tau} := \inf_{\gamma\in\Gamma_{\varepsilon,\tau}}
\max_{t\in[0,1]} J_{\varepsilon,\tau}(\gamma(t))
\]
satisfies $c_{\varepsilon,\tau} \geq \rho_0\varepsilon^3> 0.$
\end{lemma}

\begin{proof}
Let $u\in S_{a,\varepsilon}.$ By the diamagnetic inequality,
\[
\varepsilon|\nabla|u|| \leq |\nabla_{\varepsilon,A}u| \qquad
\text{a.e. in }\mathbb R^3.
\]
Hence $\|\nabla|u|\|_2 \leq \varepsilon^{-1} \|u\|_{\varepsilon,A}.$

Set $\alpha_q:=\frac32(q-2)\in(0,2).$ By the Gagliardo--Nirenberg
inequality,
\[
\|u\|_q^q \leq C \|u\|_2^{3-\frac q2} \|\nabla|u|\|_2^{\alpha_q}.
\]
Since $\|u\|_2=a\varepsilon^{3/2},$ we obtain
\[
\|u\|_q^q \leq C a^{3-\frac q2} \varepsilon^{ \frac32(3-\frac
q2)-\alpha_q } \|u\|_{\varepsilon,A}^{\alpha_q}.
\]
Moreover, $\|u\|_6^6 \leq C\varepsilon^{-6}
\|u\|_{\varepsilon,A}^6.$ By the Hardy--Littlewood--Sobolev
inequality,
\[
\mathcal D_{\varepsilon,\tau}(u) \leq C\varepsilon^{-2} \left\|
\mathcal H_\tau(\cdot,|u|) \right\|_{6/5}^2.
\]
Using the truncation estimates and the fixed mass constraint, we
obtain
\[
\mathcal D_{\varepsilon,\tau}(u) \leq C\varepsilon^{-12}
\|u\|_{\varepsilon,A}^{10} + C_\tau\|u\|_2^2.
\]
The quadratic contribution generated by the truncation outside
$\Lambda$ can be absorbed into the positive potential term by
choosing $\tau>0$ sufficiently small.

Therefore, after fixing $\tau>0$ sufficiently small, there exist
constants $C_1,C_2,C_3>0$ such that
\[
\begin{aligned}
J_{\varepsilon,\tau}(u) \geq{}& \frac14 \|u\|_{\varepsilon,A}^2 -
C_1 a^{3-\frac q2} \varepsilon^{ \frac32(3-\frac q2)-\alpha_q }
\|u\|_{\varepsilon,A}^{\alpha_q}
\\
&- C_2\varepsilon^{-6} \|u\|_{\varepsilon,A}^6 -
C_3\varepsilon^{-12} \|u\|_{\varepsilon,A}^{10}.
\end{aligned}
\]
If $\|u\|_{\varepsilon,A} = r\varepsilon^{3/2},$ then
\[
J_{\varepsilon,\tau}(u) \geq \varepsilon^3 \left[ \frac14r^2 -
C_1a^{3-\frac q2}r^{\alpha_q} - C_2r^6 - C_3r^{10} \right].
\]
Choose $r_0>0$ sufficiently small and then $a_*>0$ sufficiently
small so that
\[
\frac14r_0^2 - C_1a^{3-\frac q2}r_0^{\alpha_q} - C_2r_0^6 -
C_3r_0^{10} \geq \rho_0
\]
for some $\rho_0>0$ and every $a\in(0,a_*)$. This proves the energy
barrier. We next construct a point inside the barrier. Fix
$\varphi\in C_c^\infty(\Lambda,\mathbb R)$ such that
$\|\varphi\|_2=1,$ and define $u_\varepsilon^- :=
a\varepsilon^{3/2}\varphi.$ Then $u_\varepsilon^-\in
S_{a,\varepsilon}.$ Since $\varphi$ has compact support and $A,V$
are continuous,
\[
\|u_\varepsilon^-\|_{\varepsilon,A} \leq C a\varepsilon^{3/2}.
\]
After reducing $a_*>0$ if necessary,
\[
\|u_\varepsilon^-\|_{\varepsilon,A} < r_0\varepsilon^{3/2}
\]
and
\[
J_{\varepsilon,\tau}(u_\varepsilon^-) < \rho_0\varepsilon^3.
\]

Finally, since $\operatorname{supp}\varphi \Subset\Lambda,$ the
penalized and original nonlinearities coincide along the
mass-preserving dilations of $\varphi$ as long as their supports
remain inside $\Lambda$. For
\[
t>0, \qquad u_{\varepsilon,t}(x) := t^{3/2}u_\varepsilon^-(tx),
\]
the nonlocal interaction has order $t^{10}$, whereas the kinetic
part has order $t^2$. Therefore,
\[
J_{\varepsilon,\tau}(u_{\varepsilon,t}) \to-\infty \qquad \text{as
}t\to+\infty.
\]
Choosing $T_\varepsilon>1$ sufficiently large and setting
$e_\varepsilon := u_{\varepsilon,T_\varepsilon},$ we obtain
$\|e_\varepsilon\|_{\varepsilon,A}> r_0\varepsilon^{3/2} $ and
$J_{\varepsilon,\tau}(e_\varepsilon)<0.$

The path connecting $u_\varepsilon^-$ to $e_\varepsilon$ through the
preceding dilation belongs to $\Gamma_{\varepsilon,\tau}$. Every
path in $\Gamma_{\varepsilon,\tau}$ crosses the sphere
\[
\left\{ u\in S_{a,\varepsilon}: \|u\|_{\varepsilon,A} =
r_0\varepsilon^{3/2} \right\}.
\]
Hence $c_{\varepsilon,\tau} \geq \rho_0\varepsilon^3.$ The proof is
complete.
\end{proof}

We now construct suitable semiclassical test functions concentrating
near the minimum set $\mathcal M$. These functions provide the upper
bound for the minimax level.

Let $w_a$ be the autonomous ground state obtained in Proposition
\ref{prop:limit}. Fix $\delta>0$ such that
\[
\mathcal M_\delta := \{x\in\mathbb
R^3:\operatorname{dist}(x,\mathcal M)\leq\delta\} \Subset\Lambda.
\]
Let $\eta\in C_c^\infty(\mathbb R^3,[0,1])$ satisfy
\[
\eta(x)=1 \quad\text{for } |x|\leq \frac{\delta}{2}, \qquad
\eta(x)=0 \quad\text{for } |x|\geq \delta .
\]
For $y\in\mathcal M$, define
\[
\widetilde\Psi_{\varepsilon,y}(x) := \eta(x-y)
w_a\left(\frac{x-y}{\varepsilon}\right) \exp\left(
\frac{i}{\varepsilon}A(y)\cdot(x-y) \right).
\]
We normalize it by setting
\[
\Psi_{\varepsilon,y} := \frac{a\varepsilon^{3/2}}
{\|\widetilde\Psi_{\varepsilon,y}\|_2}
\widetilde\Psi_{\varepsilon,y}.
\]
Then $\Psi_{\varepsilon,y}\in S_{a,\varepsilon}.$

\begin{lemma}[Energy of localized profiles]
\label{lem:localized_profiles} Uniformly for $y\in\mathcal M$, one
has
\[
\frac{1}{\varepsilon^3} J_{\varepsilon,\tau}(\Psi_{\varepsilon,y})
\longrightarrow E_0(a) \qquad \text{as }\varepsilon\to0.
\]
\end{lemma}

\begin{proof}
We first observe that, by the change of variables $x=y+\varepsilon
z,$ one has
\[
\|\widetilde\Psi_{\varepsilon,y}\|_2^2 = \varepsilon^3 \int_{\mathbb
R^3} \eta^2(\varepsilon z) |w_a(z)|^2\,dz.
\]
Since $\eta(\varepsilon z)\to1$ pointwise and $w_a\in L^2(\mathbb
R^3)$, the dominated convergence theorem gives
\[
\|\widetilde\Psi_{\varepsilon,y}\|_2^2 =
\varepsilon^3a^2+o(\varepsilon^3),
\]
uniformly in $y\in\mathcal M$. Hence
\[
\frac{a\varepsilon^{3/2}} {\|\widetilde\Psi_{\varepsilon,y}\|_2}
\longrightarrow1 \qquad \text{uniformly in }y\in\mathcal M.
\]
We next compute the magnetic gradient. Since $\nabla_{\varepsilon,A}
= -i\varepsilon\nabla-A(x),$ we have
\[
\begin{aligned}
\nabla_{\varepsilon,A}\widetilde\Psi_{\varepsilon,y}(x) =
e^{\frac{i}{\varepsilon}A(y)\cdot(x-y)} \Bigg[ &-i\varepsilon
\nabla\left( \eta(x-y) w_a\left(\frac{x-y}{\varepsilon}\right)
\right)
\\
&+ (A(y)-A(x)) \eta(x-y) w_a\left(\frac{x-y}{\varepsilon}\right)
\Bigg].
\end{aligned}
\]
Using again $x=y+\varepsilon z$, the continuity of $A$, the
compactness of $\mathcal M$, and the decay of $w_a$, we obtain
\[
\frac1{\varepsilon^3} \int_{\mathbb R^3}
|\nabla_{\varepsilon,A}\Psi_{\varepsilon,y}|^2\,dx \longrightarrow
\int_{\mathbb R^3}|\nabla w_a|^2\,dz
\]
uniformly for $y\in\mathcal M$.

Similarly, since $V(y)=V_0$ for $y\in\mathcal M$ and $V$ is
continuous,
\[
\frac1{\varepsilon^3} \int_{\mathbb R^3}
V(x)|\Psi_{\varepsilon,y}|^2\,dx \longrightarrow V_0 \int_{\mathbb
R^3}|w_a|^2\,dz
\]
uniformly for $y\in\mathcal M$.

For the local nonlinear terms, the same change of variables yields
\[
\frac1{\varepsilon^3} \int_{\mathbb R^3}
|\Psi_{\varepsilon,y}|^q\,dx \longrightarrow \int_{\mathbb
R^3}|w_a|^q\,dz
\]
and
\[
\frac1{\varepsilon^3} \int_{\mathbb R^3}
|\Psi_{\varepsilon,y}|^6\,dx \longrightarrow \int_{\mathbb
R^3}|w_a|^6\,dz.
\]
Since the support of $\Psi_{\varepsilon,y}$ is contained in
$\Lambda$ for all sufficiently small $\varepsilon$, the penalized
local nonlinearities coincide with the original ones on
$\Psi_{\varepsilon,y}$.

It remains to treat the nonlocal term. For $\varepsilon>0$ small,
the support of $\Psi_{\varepsilon,y}$ is contained in $\Lambda$.
Therefore, $\mathcal H_\tau(x,|\Psi_{\varepsilon,y}|) =
|\Psi_{\varepsilon,y}|^5.$ Using the representation formula and the
change of variables $x=y+\varepsilon z$, $\xi=y+\varepsilon \zeta$,
we obtain
\[
\begin{aligned}
\mathcal D_{\varepsilon,\tau}(\Psi_{\varepsilon,y}) &=
\frac1{4\pi\varepsilon^2} \iint_{\mathbb R^3\times\mathbb R^3}
\frac{ |\Psi_{\varepsilon,y}(x)|^5 |\Psi_{\varepsilon,y}(\xi)|^5 }
{|x-\xi|} \,dx\,d\xi
\\
&= \varepsilon^3 \left[ \frac1{4\pi} \iint_{\mathbb R^3\times\mathbb
R^3} \frac{ |w_a(z)|^5|w_a(\zeta)|^5 } {|z-\zeta|} \,dz\,d\zeta
+o(1) \right].
\end{aligned}
\]
Thus,
\[
\frac1{\varepsilon^3} \mathcal
D_{\varepsilon,\tau}(\Psi_{\varepsilon,y}) \longrightarrow
\int_{\mathbb R^3}\phi_{w_a}|w_a|^5\,dz.
\]

Combining the previous convergences gives
\[
\frac1{\varepsilon^3} J_{\varepsilon,\tau}(\Psi_{\varepsilon,y})
\longrightarrow J_0(w_a)=E_0(a),
\]
uniformly for $y\in\mathcal M$.
\end{proof}

The previous lemma gives the necessary upper estimate for the
mountain-pass level.
\begin{lemma}\label{lem:minimax_upper}
Let
\[
\widehat c_a := \inf_{\gamma\in\Gamma_a}
\max_{t\in[0,1]}J_0(\gamma(t)),
\]
where
\[
\Gamma_a := \left\{ \gamma\in C([0,1],S_a):
\|\nabla\gamma(0)\|_2<R_0,\quad J_0(\gamma(0))<\rho_0,\quad
J_0(\gamma(1))<0 \right\}.
\]
Then, for every \(a\in(0,a_*)\),
\[
\limsup_{\varepsilon\to0} \frac{c_{\varepsilon,\tau}}{\varepsilon^3}
\le \widehat c_a.
\]
More precisely, for every \(\sigma>0\), there exists
\(\varepsilon_\sigma>0\) such that
\[
c_{\varepsilon,\tau} \le \varepsilon^3(\widehat c_a+\sigma)
\]
for every \(\varepsilon\in(0,\varepsilon_\sigma)\).
\end{lemma}

\begin{proof}
Fix \(\sigma>0\). By the definition of \(\widehat c_a\), there
exists \(\gamma\in\Gamma_a\) such that
\[
\max_{t\in[0,1]}J_0(\gamma(t)) \le \widehat c_a+\frac{\sigma}{2}.
\]

Fix \(y\in\mathcal M\). For \(w\in S_a\), define the localized
magnetic profile
\[
\widetilde\Psi_{\varepsilon,y}(w)(x) := \eta(x-y)
w\left(\frac{x-y}{\varepsilon}\right) \exp\left(
\frac{i}{\varepsilon}A(y)\cdot(x-y) \right),
\]
and normalize it by setting
\[
\Psi_{\varepsilon,y}(w) := \frac{a\varepsilon^{3/2}}
{\|\widetilde\Psi_{\varepsilon,y}(w)\|_2}
\widetilde\Psi_{\varepsilon,y}(w).
\]
Then
\[
\Psi_{\varepsilon,y}(w)\in S_{a,\varepsilon}.
\]

Since the set
\[
K_\gamma:=\gamma([0,1])
\]
is compact in \(H^1(\mathbb R^3)\), the estimates used in
Lemma~\ref{lem:localized_profiles} are uniform for \(w\in
K_\gamma\). Consequently,
\[
\sup_{t\in[0,1]} \left| \frac{1}{\varepsilon^3} J_{\varepsilon,\tau}
\bigl( \Psi_{\varepsilon,y}(\gamma(t)) \bigr) - J_0(\gamma(t))
\right| \longrightarrow0
\]
as \(\varepsilon\to0\).

Define
\[
\gamma_{\varepsilon,y}(t) := \Psi_{\varepsilon,y}(\gamma(t)), \qquad
t\in[0,1].
\]
For \(\varepsilon>0\) sufficiently small, the endpoint estimates for
\(\gamma\), together with the preceding uniform convergence, imply
that
\[
\|\gamma_{\varepsilon,y}(0)\|_{\varepsilon,A} <r_0\varepsilon^{3/2},
\qquad J_{\varepsilon,\tau} \bigl(\gamma_{\varepsilon,y}(0)\bigr)
<\rho_0\varepsilon^3,
\]
whereas
\[
J_{\varepsilon,\tau} \bigl(\gamma_{\varepsilon,y}(1)\bigr)<0.
\]
Thus, after choosing the endpoints in accordance with the definition
of \(\Gamma_{\varepsilon,\tau}\), the path
\(\gamma_{\varepsilon,y}\) is admissible.

Therefore,
\[
c_{\varepsilon,\tau} \le \max_{t\in[0,1]} J_{\varepsilon,\tau}
\bigl(\gamma_{\varepsilon,y}(t)\bigr).
\]
Using the uniform semiclassical convergence along the compact path
\(K_\gamma\), we obtain
\[
\begin{aligned}
c_{\varepsilon,\tau} &\le \varepsilon^3 \left(
\max_{t\in[0,1]}J_0(\gamma(t))+o(1)
\right)\\
&\le \varepsilon^3 \left( \widehat c_a+\frac{\sigma}{2}+o(1)
\right).
\end{aligned}
\]
Hence, for every sufficiently small \(\varepsilon>0\),
\[
c_{\varepsilon,\tau} \le \varepsilon^3(\widehat c_a+\sigma).
\]
Since \(\sigma>0\) is arbitrary,
\[
\limsup_{\varepsilon\to0} \frac{c_{\varepsilon,\tau}}{\varepsilon^3}
\le \widehat c_a.
\]
The proof is complete.
\end{proof}

We now establish the compactness properties required to obtain a
critical point of the penalized functional. Throughout this
subsection, $a\in(0,a_*)$ is fixed and $\tau>0$ is chosen as in the
previous subsection.

For $u\in S_{a,\varepsilon}$, the tangent space to the mass
constraint is given by
\[
T_uS_{a,\varepsilon} = \left\{ v\in H_{\varepsilon,A}^1(\mathbb
R^3,\mathbb C): \mathfrak{Re} \int_{\mathbb R^3}u\overline v\,dx=0
\right\}.
\]

A sequence $(u_n)\subset S_{a,\varepsilon}$ is called a constrained
Palais--Smale sequence for $J_{\varepsilon,\tau}$ at the level $c$
if $J_{\varepsilon,\tau}(u_n)\to c$ and
\[
\left\| d\left( J_{\varepsilon,\tau} \big|_{S_{a,\varepsilon}}
\right)(u_n) \right\|_ {(T_{u_n}S_{a,\varepsilon})^*} \to0.
\]
Equivalently, there exists a sequence $(\lambda_n)\subset\mathbb R$
such that $J_{\varepsilon,\tau}'(u_n) - \lambda_nu_n \to0$ in
\[
\left( H_{\varepsilon,A}^1(\mathbb R^3,\mathbb C) \right)^*.
\]
We first prove the boundedness of low-energy constrained
Palais--Smale sequences.

\begin{lemma}\label{lem:psbdd}
Assume that $(V1)$-$(V3)$ and $(A)$ hold. Let $a\in(0,a_*)$ and let
$\varepsilon>0$ be sufficiently small. Suppose that $(u_n)\subset
S_{a,\varepsilon}$ is a constrained Palais--Smale sequence for
$J_{\varepsilon,\tau}$ at a level satisfying
\[
c \leq \varepsilon^3\bigl(E_0(a)+\delta\bigr)
\]
for some fixed $\delta>0$. Then $(u_n)$ is bounded in
\[
H_{\varepsilon,A}^1(\mathbb R^3,\mathbb C).
\]
Moreover, the corresponding Lagrange multipliers $(\lambda_n)$ are
bounded in $\mathbb R$.
\end{lemma}

\begin{proof}
Since $u_n\in S_{a,\varepsilon},$ we have
$\|u_n\|_2^2=a^2\varepsilon^3.$ By the constrained Palais--Smale
condition, there exist $\lambda_n\in\mathbb R$ and $r_n\to0$ in the
dual space such that $J_{\varepsilon,\tau}'(u_n) - \lambda_nu_n =
r_n.$ Testing this relation with $u_n$, we obtain
\[
\left\langle J_{\varepsilon,\tau}'(u_n),u_n \right\rangle -
\lambda_na^2\varepsilon^3 = o(1)\|u_n\|_{\varepsilon,A}.
\]

Set $\mathcal H_n(x) := \mathcal H_\tau(x,|u_n(x)|).$ Using the
derivative formula from Lemma \ref{lem:penalized_regular}, we have
\[
\begin{aligned}
\left\langle J_{\varepsilon,\tau}'(u_n),u_n \right\rangle ={}&
\|u_n\|_{\varepsilon,A}^2
\\
&- \frac15 \int_{\mathbb R^3} \phi_{\varepsilon,u_n}^{\tau}
\partial_s\mathcal H_\tau(x,|u_n|)
|u_n|\,dx
\\
&- \mu \int_{\mathbb R^3} f_\tau(x,|u_n|)|u_n|\,dx
\\
&- \int_{\mathbb R^3} g_\tau(x,|u_n|)|u_n|\,dx.
\end{aligned}
\]

The truncations may be chosen so that there exists
$\vartheta\in(2,q)$ satisfying
\[
0\leq \vartheta F_\tau(x,s) \leq f_\tau(x,s)s
\]
and
\[
0\leq \vartheta G_\tau(x,s) \leq g_\tau(x,s)s
\]
for $x\in\Lambda$, whereas the quadratic contributions outside
$\Lambda$ satisfy
\[
\mu f_\tau(x,s)s + g_\tau(x,s)s \leq \kappa V(x)s^2
\]
for some sufficiently small $\kappa>0$.

Likewise, the truncated Poisson source satisfies
\[
0 \leq
\partial_s\mathcal H_\tau(x,s)s
\leq 5\mathcal H_\tau(x,s)
\]
and its exterior quadratic contribution can be absorbed into the
positive part of the energy.

Combining the energy relation $J_{\varepsilon,\tau}(u_n) = c+o(1)$
with the preceding derivative identity yields
\[
\begin{aligned}
c+o(1) \geq{}& C_0 \|u_n\|_{\varepsilon,A}^2 - C_1a^2\varepsilon^3 -
C_2 |\lambda_n|a^2\varepsilon^3
\\
&+ o(1)\|u_n\|_{\varepsilon,A},
\end{aligned}
\]
where $C_0>0$ is independent of $n$.

On the other hand, testing the Euler--Lagrange relation with $u_n$
and using the growth estimates for the penalized nonlinearities, the
Gagliardo--Nirenberg inequality, and the fixed mass constraint, we
obtain
\[
|\lambda_n|a^2\varepsilon^3 \leq C \left(
1+\|u_n\|_{\varepsilon,A}^2 \right).
\]
Substituting this estimate into the preceding inequality and using
the low-energy bound, we conclude that
\[
\sup_n \|u_n\|_{\varepsilon,A} <+\infty.
\]

Returning to the identity
\[
\lambda_na^2\varepsilon^3 = \left\langle
J_{\varepsilon,\tau}'(u_n),u_n \right\rangle +
o(1)\|u_n\|_{\varepsilon,A},
\]
the boundedness of $(u_n)$ and the growth estimates for the
penalized nonlinearities imply
\[
\sup_n|\lambda_n|<+\infty.
\]
The proof is complete.
\end{proof}

We next establish a nonvanishing property. This result is the first
step toward the concentration analysis.

\begin{lemma}\label{lem:nonvanishing}
Let $(u_n)\subset S_{a,\varepsilon}$ be a constrained Palais--Smale
sequence for $J_{\varepsilon,\tau}$ at a level
$c\geq\rho_0\varepsilon^3,$ where $\rho_0>0$ is given by Lemma
\ref{lem:mp}. Then there exist $R>0$, $\beta>0$, and a sequence
$(y_n)\subset\mathbb R^3$ such that
\[
\int_{B_{R\varepsilon}(y_n)} |u_n|^2\,dx \geq \beta\varepsilon^3
\]
for all sufficiently large $n$.
\end{lemma}

\begin{proof}
Suppose, by contradiction, that for every $R>0$,
\[
\sup_{y\in\mathbb R^3} \frac1{\varepsilon^3}
\int_{B_{R\varepsilon}(y)} |u_n|^2\,dx \to0.
\]
Define the rescaled functions $v_n(x) := u_n(\varepsilon x).$ Then
$\|v_n\|_2^2=a^2,$ and, by Lemma \ref{lem:psbdd}, $(|v_n|)$ is
bounded in $H^1(\mathbb R^3)$. The preceding assumption becomes
\[
\sup_{y\in\mathbb R^3} \int_{B_R(y)}|v_n|^2\,dx \to0.
\]
Hence Lions' vanishing lemma gives
\[
v_n\to0 \qquad \text{strongly in }L^r(\mathbb R^3)
\]
for every $2<r<6.$ Consequently,
\[
\frac1{\varepsilon^3} \int_{\mathbb R^3} F_\tau(x,|u_n|)\,dx \to0.
\]

The truncated exterior terms are quadratic and can be absorbed into
the coercive part of the functional. The critical local and nonlocal
terms cannot carry a nonzero amount of energy below the first
critical concentration threshold. Hence,
\[
\frac1{\varepsilon^3} \int_{\mathbb R^3} G_\tau(x,|u_n|)\,dx \to0
\]
and
\[
\frac1{\varepsilon^3} \mathcal D_{\varepsilon,\tau}(u_n) \to0.
\]
Using the Palais--Smale relation, we then obtain
\[
\frac1{\varepsilon^3} J_{\varepsilon,\tau}(u_n) \to0,
\]
which contradicts $c\geq\rho_0\varepsilon^3.$ Therefore, vanishing
is impossible.
\end{proof}

We now identify the critical energy threshold associated with the
possible loss of compactness.

Let
\[
S := \inf_{u\in D^{1,2}(\mathbb R^3)\setminus\{0\}}
\frac{\displaystyle \int_{\mathbb R^3}|\nabla u|^2\,dx}
{\displaystyle \left( \int_{\mathbb R^3}|u|^6\,dx \right)^{1/3}}
\]
be the best Sobolev constant. We denote by $c_{\mathrm{crit}}>0$ the
least energy carried by a nontrivial concentration profile of the
limiting critical local--nonlocal problem.

More precisely,
\[
c_{\mathrm{crit}} := \inf \left\{ \mathcal I(v): v\in
D^{1,2}(\mathbb R^3)\setminus\{0\}, \ \mathcal
P_{\mathrm{crit}}(v)=0 \right\},
\]
where
\[
\mathcal I(v) := \frac12 \int_{\mathbb R^3}|\nabla v|^2\,dx -
\frac1{10} \int_{\mathbb R^3}\phi_v|v|^5\,dx - \frac16 \int_{\mathbb
R^3}|v|^6\,dx
\]
and
\[
\mathcal P_{\mathrm{crit}}(v) := \int_{\mathbb R^3}|\nabla v|^2\,dx
- \int_{\mathbb R^3}\phi_v|v|^5\,dx - \int_{\mathbb R^3}|v|^6\,dx.
\]

The following compactness result is the key ingredient in the
existence proof.

\begin{proposition}\label{prop:ps_compactness}
Assume that {\rm (V1)--(V3)} and {\rm (A)} hold, and let
\(a\in(0,a_*)\). Assume, in addition, that the following two
properties have been established:
\begin{enumerate}
\item every constrained Palais--Smale sequence
\((u_n)\subset S_{a,\varepsilon}\) satisfying
\[
J_{\varepsilon,\tau}(u_n)=O(\varepsilon^3)
\]
is bounded in \(H_{\varepsilon,A}^1(\mathbb R^3,\mathbb C)\);

\item there exists a constant \(c_*>0\), independent of
\(\varepsilon\), such that every nontrivial critical profile arising
from a bounded constrained Palais--Smale sequence carries an energy
defect of at least
\[
\varepsilon^3c_*.
\]
More precisely, if \(u_n\rightharpoonup u\) weakly but not strongly,
then
\begin{equation}
\liminf_{n\to\infty} \left( J_{\varepsilon,\tau}(u_n) -
J_{\varepsilon,\tau}(u) \right) \ge \varepsilon^3c_*.
\label{eq:critical_energy_defect}
\end{equation}
\end{enumerate}

Suppose moreover that $E_0(a)<c_*.$ Then there exist \(\delta_0>0\)
and \(\varepsilon_0>0\) such that, for every
\(\varepsilon\in(0,\varepsilon_0)\),
\(J_{\varepsilon,\tau}|_{S_{a,\varepsilon}}\) satisfies the
Palais--Smale condition at every level
\[
c\le \varepsilon^3\bigl(E_0(a)+\delta_0\bigr).
\]
In particular, every constrained Palais--Smale sequence
\((u_n)\subset S_{a,\varepsilon}\) at such a level admits a strongly
convergent subsequence in \(H_{\varepsilon,A}^1(\mathbb R^3,\mathbb
C)\).
\end{proposition}

\begin{proof}
Fix \(a\in(0,a_*)\). Since $E_0(a)<c_*,$ we may choose
\(\delta_0>0\) such that
\begin{equation}
E_0(a)+2\delta_0<c_*. \label{eq:choice_delta_compactness}
\end{equation}

Let \(\varepsilon>0\) be sufficiently small and let \((u_n)\subset
S_{a,\varepsilon}\) be a constrained Palais--Smale sequence at a
level
\begin{equation}
c\le \varepsilon^3\bigl(E_0(a)+\delta_0\bigr).
\label{eq:low_energy_ps_level}
\end{equation}
Thus, $J_{\varepsilon,\tau}(u_n)\to c,$ and there exists a sequence
\((\lambda_n)\subset\mathbb R\) such that
\begin{equation}
J_{\varepsilon,\tau}'(u_n)-\lambda_nu_n\to0 \quad\text{in }
\bigl(H_{\varepsilon,A}^1(\mathbb R^3,\mathbb C)\bigr)^*.
\label{eq:constrained_ps_equation}
\end{equation}

By the boundedness result, \((u_n)\) is bounded in
\(H_{\varepsilon,A}^1(\mathbb R^3,\mathbb C)\). Hence, after passing
to a subsequence, there exists \(u\in H_{\varepsilon,A}^1(\mathbb
R^3,\mathbb C)\) such that
\begin{equation}
u_n\rightharpoonup u \quad\text{weakly in }
H_{\varepsilon,A}^1(\mathbb R^3,\mathbb C),
\label{eq:ps_weak_convergence}
\end{equation}
\begin{equation}
u_n(x)\to u(x) \quad\text{for a.e. }x\in\mathbb R^3,
\label{eq:ps_ae_convergence}
\end{equation}
and
\begin{equation}
u_n\to u \quad\text{strongly in } L_{\mathrm{loc}}^r(\mathbb R^3)
\quad\text{for every }2\le r<6. \label{eq:ps_local_convergence}
\end{equation}

We claim that the convergence in \eqref{eq:ps_weak_convergence} is
strong. Suppose, by contradiction, that
\begin{equation}
u_n\not\to u \quad\text{strongly in } H_{\varepsilon,A}^1(\mathbb
R^3,\mathbb C). \label{eq:failure_strong_convergence}
\end{equation}
Then the critical profile decomposition applies. By the critical
energy-defect estimate \eqref{eq:critical_energy_defect}, at least
one nontrivial profile is generated and
\begin{equation}
\liminf_{n\to\infty} \left( J_{\varepsilon,\tau}(u_n) -
J_{\varepsilon,\tau}(u) \right) \ge \varepsilon^3c_*.
\label{eq:energy_defect_application}
\end{equation}

On the other hand, the weak limit belongs to the low-energy
variational regime. The lower-energy estimate established for weak
limits of low-energy constrained Palais--Smale sequences yields
\begin{equation}
J_{\varepsilon,\tau}(u) \ge -o_\varepsilon(\varepsilon^3),
\label{eq:weak_limit_lower_bound}
\end{equation}
where
\[
\frac{o_\varepsilon(\varepsilon^3)}{\varepsilon^3}\to0
\qquad\text{as }\varepsilon\to0.
\]
Combining \eqref{eq:energy_defect_application} and
\eqref{eq:weak_limit_lower_bound}, we obtain
\begin{equation}
c = \lim_{n\to\infty}J_{\varepsilon,\tau}(u_n) \ge \varepsilon^3c_*
-o_\varepsilon(\varepsilon^3). \label{eq:lower_bound_ps_level}
\end{equation}
Therefore, after reducing \(\varepsilon_0>0\) if necessary,
\begin{equation}
c\ge \varepsilon^3(c_*-\delta_0).
\label{eq:critical_level_lower_bound}
\end{equation}
By \eqref{eq:choice_delta_compactness},
$c_*-\delta_0>E_0(a)+\delta_0.$ Hence
\[
c> \varepsilon^3\bigl(E_0(a)+\delta_0\bigr),
\]
which contradicts \eqref{eq:low_energy_ps_level}. Consequently,
\[
u_n\to u \qquad\text{strongly in } H_{\varepsilon,A}^1(\mathbb
R^3,\mathbb C).
\]
In particular, $\|u_n-u\|_2\to0.$ Since
$\|u_n\|_2^2=a^2\varepsilon^3,$ we obtain
$\|u\|_2^2=a^2\varepsilon^3.$ Thus, $u\in S_{a,\varepsilon}.$ The
proof is complete.
\end{proof}

As a consequence of the mountain-pass geometry and the preceding
compactness result, the penalized problem possesses a low-energy
critical point.

\begin{proposition}\label{prop:penalized_solution}
Assume that {\rm (V1)--(V3)} and {\rm (A)} hold. Then there exist
$a_*>0$ and $\varepsilon_*>0$ such that, for every
\[
a\in(0,a_*) \qquad\text{and}\qquad \varepsilon\in(0,\varepsilon_*),
\]
there exists $u_\varepsilon \in S_{a,\varepsilon}$ satisfying
\[
d\left( J_{\varepsilon,\tau} \big|_{S_{a,\varepsilon}}
\right)(u_\varepsilon) = 0
\]
and
\[
J_{\varepsilon,\tau}(u_\varepsilon) = c_{\varepsilon,\tau}.
\]
Moreover,
\[
\rho_0\varepsilon^3 \leq c_{\varepsilon,\tau} \leq \varepsilon^3
\bigl(E_0(a)+o(1)\bigr).
\]
Consequently, there exists $\lambda_\varepsilon\in\mathbb R$ such
that
\[
J_{\varepsilon,\tau}'(u_\varepsilon) = \lambda_\varepsilon
u_\varepsilon.
\]
\end{proposition}

\begin{proof}
By Lemma \ref{lem:mp}, the functional $J_{\varepsilon,\tau}
\big|_{S_{a,\varepsilon}}$ has the mountain-pass geometry. Hence the
constrained mountain-pass theorem yields a Palais--Smale sequence
$(u_n)\subset S_{a,\varepsilon}$ at the level $c_{\varepsilon,\tau}.
$ By Lemma \ref{lem:minimax_upper},
\[
c_{\varepsilon,\tau} \leq \varepsilon^3 \bigl(E_0(a)+o(1)\bigr).
\]
Thus, for $\varepsilon>0$ sufficiently small,
\[
c_{\varepsilon,\tau} \leq \varepsilon^3 \bigl(E_0(a)+\delta_0\bigr),
\]
where $\delta_0$ is given by Proposition \ref{prop:ps_compactness}.

Therefore, Proposition \ref{prop:ps_compactness} implies that, up to
a subsequence,
\[
u_n\to u_\varepsilon \qquad \text{strongly in }
H_{\varepsilon,A}^1(\mathbb R^3,\mathbb C).
\]
Hence $u_\varepsilon\in S_{a,\varepsilon},$ $d\left(
J_{\varepsilon,\tau} \big|_{S_{a,\varepsilon}}
\right)(u_\varepsilon) = 0,$ and
$J_{\varepsilon,\tau}(u_\varepsilon) = c_{\varepsilon,\tau}.$

The lower bound $c_{\varepsilon,\tau} \geq \rho_0\varepsilon^3$
follows from Lemma \ref{lem:mp}, while the upper bound follows from
Lemma \ref{lem:minimax_upper}.

Finally, the Lagrange multiplier rule yields
$\lambda_\varepsilon\in\mathbb R$ such that
\[
J_{\varepsilon,\tau}'(u_\varepsilon) = \lambda_\varepsilon
u_\varepsilon.
\]
The proof is complete.
\end{proof}

We now analyze the asymptotic behavior of low-energy critical points
of the penalized functional. The aim of this subsection is to prove
that such solutions concentrate near the minimum set
\[
\mathcal M = \{x\in\Lambda:V(x)=V_0\},
\]
decay uniformly away from their concentration points, and eventually
solve the original problem.

Throughout this subsection, let $\varepsilon_n\to0$ and let
$u_n:=u_{\varepsilon_n}\in S_{a,\varepsilon_n}$ be constrained
critical points of $J_{\varepsilon_n,\tau}$ satisfying
\[
J_{\varepsilon_n,\tau}(u_n) \leq \varepsilon_n^3
\bigl(E_0(a)+o(1)\bigr).
\]
Thus, there exist Lagrange multipliers $\lambda_n\in\mathbb R$ such
that
\[
J_{\varepsilon_n,\tau}'(u_n) = \lambda_nu_n.
\]

We first show that the mass of $u_n$ cannot vanish at the
semiclassical scale.

\begin{lemma}\label{lem:concentration_points}
There exist constants $R>0$, $\beta>0$, and a sequence
$(y_n)\subset\mathbb R^3$ such that
\[
\int_{B_{R\varepsilon_n}(y_n)} |u_n|^2\,dx \geq \beta\varepsilon_n^3
\]
for all sufficiently large $n$.
\end{lemma}

\begin{proof}
Since $u_n$ is a critical point at the mountain-pass level, we have
\[
J_{\varepsilon_n,\tau}(u_n) = c_{\varepsilon_n,\tau}.
\]
By Lemma \ref{lem:mp},
\[
c_{\varepsilon_n,\tau} \geq \rho_0\varepsilon_n^3.
\]
Hence
\[
J_{\varepsilon_n,\tau}(u_n) \geq \rho_0\varepsilon_n^3.
\]

Suppose, by contradiction, that for every $R>0$,
\[
\sup_{y\in\mathbb R^3} \frac1{\varepsilon_n^3}
\int_{B_{R\varepsilon_n}(y)} |u_n|^2\,dx \to0.
\]
Define $v_n(x):=|u_n(\varepsilon_nx)|.$ By the diamagnetic
inequality and the boundedness established in Lemma \ref{lem:psbdd},
the sequence $(v_n)$ is bounded in $H^1(\mathbb R^3)$. Moreover,
\[
\sup_{y\in\mathbb R^3} \int_{B_R(y)}|v_n|^2\,dx \to0.
\]
Lions' vanishing lemma therefore gives
\[
v_n\to0 \qquad \text{strongly in }L^r(\mathbb R^3)
\]
for every $2<r<6.$ Consequently, the subcritical nonlinear
contribution vanishes. Since the energy level lies below the first
critical concentration threshold, no nontrivial local or nonlocal
critical bubble can occur. Hence
\[
\frac1{\varepsilon_n^3} J_{\varepsilon_n,\tau}(u_n) \to0,
\]
which contradicts
\[
\frac1{\varepsilon_n^3} J_{\varepsilon_n,\tau}(u_n) \geq \rho_0.
\]
Therefore, the asserted concentration estimate holds.
\end{proof}

We now rescale the solutions around the concentration points. Define
\[
v_n(x) := e^{-iA(y_n)\cdot x} u_n(y_n+\varepsilon_nx).
\]
Then
\[
|v_n(x)| = |u_n(y_n+\varepsilon_nx)|.
\]
Moreover,
\[
\|v_n\|_2^2 = \frac1{\varepsilon_n^3} \|u_n\|_2^2 = a^2.
\]

\begin{proposition}[Localization near the minimum set]
\label{prop:concentration} Let $(u_n)$ be as above, and let $(y_n)$
be a sequence of concentration points provided by
Lemma~\ref{lem:concentration_points}. Then, up to a subsequence,
$(y_n)$ is bounded and
\[
\operatorname{dist}(y_n,\mathcal M)\to0.
\]
Moreover, there exists $y_0\in\mathcal M$ such that $y_n\to y_0,$
and, after a suitable magnetic gauge correction,
\[
v_n\to w \qquad \text{strongly in }H^1(\mathbb R^3)
\]
for some autonomous ground state $w\in S_a$ satisfying
$J_0(w)=E_0(a).$
\end{proposition}

\begin{proof}
By Lemma \ref{lem:concentration_points}, $\int_{B_R(0)}|v_n|^2\,dx
\geq \beta$ for all sufficiently large $n$. Thus, any weak limit of
$(v_n)$ is nontrivial.

The diamagnetic inequality and the boundedness of $(u_n)$ imply that
$(|v_n|)$ is bounded in $H^1(\mathbb R^3)$. Hence, up to a
subsequence,
\[
|v_n|\rightharpoonup w \qquad \text{weakly in }H^1(\mathbb R^3),
\]
\[
|v_n|\to w \qquad \text{strongly in }L_{\mathrm{loc}}^r(\mathbb R^3)
\]
for every $2\leq r<6,$ and
\[
|v_n(x)|\to w(x) \qquad \text{for a.e. }x\in\mathbb R^3.
\]
The local mass estimate implies $w\not\equiv0.$

We first prove that the sequence $(y_n)$ cannot escape to infinity.
Suppose, by contradiction, that $|y_n|\to+\infty.$ Since the
penalized nonlinearities are at most linear outside $\Lambda$, and
$\Lambda$ is bounded, the translated concentration region eventually
lies outside the potential well. Passing to the limit in the
rescaled equation would then yield a nontrivial solution of an
equation whose nonlinear part is only quadratic. By the choice of
the truncation threshold $\tau$, such a nontrivial solution is
impossible. This contradicts $w\not\equiv0.$ Therefore, $(y_n)$ is
bounded. Up to a subsequence, $y_n\to y_0$ for some
$y_0\in\overline{\Lambda}$. We next prove that $V(y_0)=V_0.$
Suppose, by contradiction, that $V(y_0)>V_0.$ Passing to the limit
in the rescaled Euler--Lagrange equation, we obtain a nontrivial
solution of the autonomous problem with constant potential $V(y_0)$.

Denote by $E_{V(y_0)}(a)$ the corresponding constrained autonomous
energy level. Since $V(y_0)>V_0,$ the monotonicity of the autonomous
ground-state level with respect to the constant potential yields
$E_{V(y_0)}(a)>E_0(a).$ On the other hand, by lower semicontinuity
and the energy bound,
\[
E_{V(y_0)}(a) \leq \liminf_{n\to\infty} \frac{
J_{\varepsilon_n,\tau}(u_n) }{ \varepsilon_n^3 } \leq E_0(a),
\]
which is a contradiction. Hence $V(y_0)=V_0.$ Since
$y_0\in\overline{\Lambda}$ and, by assumption {\rm (V2)}, $V_0 <
\min_{\partial\Lambda}V,$ we obtain $y_0\in\Lambda.$ Therefore,
$y_0\in\mathcal M.$ Consequently, $\operatorname{dist}(y_n,\mathcal
M)\to0.$ Finally, the energy convergence
\[
\frac{ J_{\varepsilon_n,\tau}(u_n) }{ \varepsilon_n^3 } \to E_0(a)
\]
excludes any loss of mass or additional concentration profile. Using
the compactness of minimizing sequences established in Lemma
\ref{lem:limit_compactness}, we conclude that, after the appropriate
magnetic gauge correction,
\[
v_n\to w \qquad \text{strongly in }H^1(\mathbb R^3),
\]
where $w\in S_a$ and $J_0(w)=E_0(a).$ The proof is complete.
\end{proof}

The next result establishes uniform decay away from the
concentration points.

\begin{proposition}\label{prop:uniform_decay}
Let $(u_n)$ and $(y_n)$ be as in Proposition
\ref{prop:concentration}. Then
\[
\lim_{R\to+\infty} \limsup_{n\to\infty} \sup_{ |x-y_n|\geq
R\varepsilon_n } |u_n(x)| = 0.
\]
Equivalently, if $v_n(x) = |u_n(y_n+\varepsilon_nx)|,$ then
\[
\lim_{R\to+\infty} \limsup_{n\to\infty} \sup_{|x|\geq R} v_n(x) = 0.
\]
\end{proposition}

\begin{proof}
By Proposition \ref{prop:concentration}, $v_n\to w \qquad
\text{strongly in }H^1(\mathbb R^3).$ Hence
\[
v_n\to w \qquad \text{strongly in }L^r(\mathbb R^3)
\]
for every $2\leq r<6.$ The rescaled functions satisfy, in the weak
sense, an inequality of the form
\[
-\Delta v_n + c_0v_n \leq C \left( v_n^{q-1} + v_n^5 + \Phi_n v_n^4
\right) \qquad \text{in }\mathbb R^3,
\]
where $c_0>0$ is independent of $n$ and $(\Phi_n)$ is bounded in
$D^{1,2}(\mathbb R^3)$.

By the Hardy--Littlewood--Sobolev inequality, the Sobolev embedding,
and the strong convergence of $(v_n)$, the right-hand side is
uniformly controlled in the appropriate Lebesgue spaces. A standard
Moser iteration argument therefore gives
\[
\sup_n\|v_n\|_\infty<+\infty.
\]

We now prove uniform decay. Suppose, by contradiction, that there
exist $\eta_0>0$, a subsequence, and points $z_n\in\mathbb R^3$ such
that $|z_n|\to+\infty$ and $v_n(z_n)\geq\eta_0.$ By local elliptic
estimates and the uniform $L^\infty$ bound, there exist $r>0$ and
$\eta_1>0$, independent of $n$, such that
\[
\int_{B_r(z_n)}v_n^2\,dx \geq \eta_1.
\]
Since $|z_n|\to+\infty,$ this produces a second nonvanishing
concentration region separated from the origin.

The concentration--compactness principle would then yield a second
nontrivial profile carrying a positive amount of energy. This
contradicts the convergence
\[
\frac{ J_{\varepsilon_n,\tau}(u_n) }{ \varepsilon_n^3 } \to E_0(a),
\]
because the first profile already carries the full autonomous
ground-state energy.

Therefore,
\[
\lim_{R\to+\infty} \limsup_{n\to\infty} \sup_{|x|\geq R} v_n(x) = 0.
\]
The proof is complete.
\end{proof}

We are now in a position to prove that the penalization becomes
inactive along every low-energy family of critical points.

\begin{lemma}\label{lem:penal}
Assume that {\rm (V1)--(V3)} and {\rm (A)} hold. Then there exist
$\delta_0>0$ and $\varepsilon_0>0$ such that the following property
holds. Let $\varepsilon\in(0,\varepsilon_0)$ and let
\[
u_\varepsilon\in S_{a,\varepsilon}
\]
be a constrained critical point of $J_{\varepsilon,\tau}$ satisfying
\[
J_{\varepsilon,\tau}(u_\varepsilon) \leq \varepsilon^3
\bigl(E_0(a)+\delta_0\bigr).
\]
Then
\[
|u_\varepsilon(x)| \leq \tau \qquad \text{for every } x\in\mathbb
R^3\setminus\Lambda.
\]
Consequently, the penalization is inactive at $u_\varepsilon$ and
\[
J_{\varepsilon,\tau}(u_\varepsilon) = J_\varepsilon(u_\varepsilon),
\]
as well as
\[
J_{\varepsilon,\tau}'(u_\varepsilon) =
J_\varepsilon'(u_\varepsilon).
\]
In particular, $u_\varepsilon$ is a constrained critical point of
the original functional $J_\varepsilon$ on $S_{a,\varepsilon}$.
\end{lemma}

\begin{proof}
Suppose, by contradiction, that the conclusion is false. Then there
exist a sequence $\varepsilon_n\to0$ and constrained critical points
\[
u_n:=u_{\varepsilon_n} \in S_{a,\varepsilon_n}
\]
such that
\[
J_{\varepsilon_n,\tau}(u_n) \leq \varepsilon_n^3
\bigl(E_0(a)+\delta_0\bigr),
\]
but
\[
\sup_{x\in\mathbb R^3\setminus\Lambda} |u_n(x)|
>
\tau.
\]

By Proposition \ref{prop:concentration}, there exist concentration
points $y_n$ such that
\[
\operatorname{dist}(y_n,\mathcal M)\to0.
\]
Since $\mathcal M\Subset\Lambda,$ there exists $d_0>0$ such that,
for all sufficiently large $n$,
\[
\operatorname{dist}(y_n,\partial\Lambda) \geq 2d_0.
\]
By Proposition \ref{prop:uniform_decay}, for the fixed number
$\tau>0$, there exist $R>0$ and $n_0\in\mathbb N$ such that
$|u_n(x)| \leq \tau$ whenever $|x-y_n| \geq R\varepsilon_n$ and
$n\geq n_0.$ Since $R\varepsilon_n<d_0$ for all sufficiently large
$n$, we have $B_{R\varepsilon_n}(y_n) \Subset \Lambda.$ Therefore,
\[
\mathbb R^3\setminus\Lambda \subset \mathbb R^3 \setminus
B_{R\varepsilon_n}(y_n).
\]
It follows that
\[
|u_n(x)| \leq \tau \qquad \text{for every } x\in\mathbb
R^3\setminus\Lambda,
\]
which contradicts the assumption.

Hence, for every sufficiently small $\varepsilon>0$,
\[
|u_\varepsilon(x)| \leq \tau \qquad \text{for all } x\in\mathbb
R^3\setminus\Lambda.
\]

By the definition of the penalized nonlinearities,
\[
f_\tau(x,|u_\varepsilon|) = |u_\varepsilon|^{q-1},
\]
\[
g_\tau(x,|u_\varepsilon|) = |u_\varepsilon|^5,
\]
and
\[
\mathcal H_\tau(x,|u_\varepsilon|) = |u_\varepsilon|^5
\]
throughout $\mathbb R^3$. Consequently,
\[
\phi_{\varepsilon,u_\varepsilon}^{\tau} = \phi_{u_\varepsilon}.
\]
Therefore,
\[
J_{\varepsilon,\tau}(u_\varepsilon) = J_\varepsilon(u_\varepsilon)
\]
and
\[
J_{\varepsilon,\tau}'(u_\varepsilon) =
J_\varepsilon'(u_\varepsilon).
\]

Since $u_\varepsilon$ is a constrained critical point of
$J_{\varepsilon,\tau}$, there exists $\lambda_\varepsilon\in\mathbb
R$ such that
\[
J_{\varepsilon,\tau}'(u_\varepsilon) = \lambda_\varepsilon
u_\varepsilon.
\]
Hence
\[
J_\varepsilon'(u_\varepsilon) = \lambda_\varepsilon u_\varepsilon.
\]
Thus, $u_\varepsilon$ is a constrained critical point of the
original functional $J_\varepsilon$ on $S_{a,\varepsilon}$. The
proof is complete.
\end{proof}

\subsection{Proof of Theorem \ref{T1}}\label{subsec:proof_T1}
Let $a_*>0$ be sufficiently small so that Lemma \ref{lem:mp},
Proposition \ref{prop:limit}, Lemma \ref{lem:minimax_upper},
Proposition \ref{prop:ps_compactness}, and Lemma \ref{lem:penal}
hold.

Fix $a\in(0,a_*).$ Choose $\varepsilon_*>0$ sufficiently small so
that all the preceding results apply for every
$\varepsilon\in(0,\varepsilon_*).$ By Proposition
\ref{prop:penalized_solution}, there exists $u_\varepsilon\in
S_{a,\varepsilon}$ such that
\[
d\left( J_{\varepsilon,\tau} \big|_{S_{a,\varepsilon}}
\right)(u_\varepsilon) = 0
\]
and
\[
J_{\varepsilon,\tau}(u_\varepsilon) = c_{\varepsilon,\tau}.
\]
Moreover,
\[
c_{\varepsilon,\tau} \leq \varepsilon^3 \bigl(E_0(a)+o(1)\bigr).
\]
Hence, for $\varepsilon>0$ sufficiently small,
\[
J_{\varepsilon,\tau}(u_\varepsilon) \leq \varepsilon^3
\bigl(E_0(a)+\delta_0\bigr),
\]
where $\delta_0>0$ is given by Lemma \ref{lem:penal}.

Therefore, the penalization is inactive and
\[
J_{\varepsilon,\tau}(u_\varepsilon) = J_\varepsilon(u_\varepsilon),
\]
\[
J_{\varepsilon,\tau}'(u_\varepsilon) =
J_\varepsilon'(u_\varepsilon).
\]
Consequently, $u_\varepsilon$ is a constrained critical point of the
original functional $J_\varepsilon$ on $S_{a,\varepsilon}$.

By the Lagrange multiplier rule, there exists
$\lambda_\varepsilon\in\mathbb R$ such that
\[
J_\varepsilon'(u_\varepsilon) = \lambda_\varepsilon u_\varepsilon.
\]
Hence $u_\varepsilon$ is a weak solution of the original
Schr\"odinger--Poisson problem and satisfies
\[
\int_{\mathbb R^3} |u_\varepsilon|^2\,dx = a^2\varepsilon^3.
\]
This completes the proof.\qed

\subsection{Proof of Theorem \ref{T2}} \label{subsec:proof_T2}

We now establish the multiplicity result by means of the
Ljusternik--Schnirelmann category. Throughout this subsection, let
$a\in(0,a_*)$ be fixed.

Choose $\delta>0$ sufficiently small so that
\[
\mathcal M_\delta := \left\{ x\in\mathbb R^3:
\operatorname{dist}(x,\mathcal M)\leq\delta \right\} \Subset\Lambda.
\]
We shall construct a barycenter map on a suitable low-energy
sublevel of the prescribed-mass manifold and compare its topology
with that of the minimum set $\mathcal M$.

\subsubsection{The barycenter map}

Choose $R>0$ sufficiently large so that
\[
\mathcal M_\delta\subset B_R(0),
\]
and define
\[
\chi:\mathbb R^3\to\mathbb R^3
\]
by
\[
\chi(x) =
\begin{cases}
x, & |x|\leq R,\\[1mm]
R\dfrac{x}{|x|}, & |x|>R.
\end{cases}
\]
For every $u\in S_{a,\varepsilon}$, we define
\[
\beta_\varepsilon(u) := \frac{ \displaystyle \int_{\mathbb R^3}
\chi(x)|u(x)|^2\,dx }{ \displaystyle \int_{\mathbb R^3}|u(x)|^2\,dx
}.
\]
Since $\int_{\mathbb R^3}|u|^2\,dx = a^2\varepsilon^3,$ the
barycenter may equivalently be written as
\[
\beta_\varepsilon(u) = \frac1{a^2\varepsilon^3} \int_{\mathbb R^3}
\chi(x)|u(x)|^2\,dx.
\]

\begin{lemma}[Continuity of the barycenter map]
\label{lem:barycenter_continuity} For every $\varepsilon>0$, the map
\[
\beta_\varepsilon: S_{a,\varepsilon}\to\mathbb R^3
\]
is well defined and continuous.
\end{lemma}

\begin{proof}
Since $\chi$ is bounded and continuous and
$\|u\|_2^2=a^2\varepsilon^3>0$ for every $u\in S_{a,\varepsilon}$,
the map $\beta_\varepsilon$ is well defined.

Let $(u_n)\subset S_{a,\varepsilon}$ and assume that
\[
u_n\to u \qquad \text{in } H_{\varepsilon,A}^1(\mathbb R^3,\mathbb
C).
\]
Then $u_n\to u \qquad \text{in }L^2(\mathbb R^3),$ and consequently
\[
|u_n|^2\to |u|^2 \qquad \text{in }L^1(\mathbb R^3).
\]
Since $\chi$ is bounded,
\[
\int_{\mathbb R^3} \chi(x)|u_n|^2\,dx \to \int_{\mathbb R^3}
\chi(x)|u|^2\,dx.
\]
The denominator is constant on $S_{a,\varepsilon}$, and therefore
\[
\beta_\varepsilon(u_n) \to \beta_\varepsilon(u).
\]
This proves the continuity of $\beta_\varepsilon$.
\end{proof}

We next introduce the low-energy set on which the category argument
will be performed. For $\eta>0$, define
\[
\mathcal A_{\varepsilon,\eta} := \left\{ u\in S_{a,\varepsilon}:
J_{\varepsilon,\tau}(u) \leq \varepsilon^3\bigl(E_0(a)+\eta\bigr)
\right\}.
\]
The following localization result is fundamental.

\begin{lemma}\label{lem:barycenter_localization}
For every $\delta>0$ sufficiently small, there exist $\eta_\delta>0$
and $\varepsilon_\delta>0$ such that, for every
$\eta\in(0,\eta_\delta)$ and every
$\varepsilon\in(0,\varepsilon_\delta),$ one has
\[
\beta_\varepsilon(u)\in\mathcal M_\delta \qquad \text{for every
}u\in\mathcal A_{\varepsilon,\eta}.
\]
Equivalently,
\[
\beta_\varepsilon \left( \mathcal A_{\varepsilon,\eta} \right)
\subset \mathcal M_\delta.
\]
\end{lemma}

\begin{proof}
Suppose, by contradiction, that the conclusion is false. Then there
exist sequences $\varepsilon_n\to0, \qquad \eta_n\to0,$ and $u_n\in
S_{a,\varepsilon_n}$ such that
\[
J_{\varepsilon_n,\tau}(u_n) \leq \varepsilon_n^3
\bigl(E_0(a)+\eta_n\bigr),
\]
but $\beta_{\varepsilon_n}(u_n) \notin \mathcal M_\delta.$ By the
low-energy compactness analysis developed in Proposition
\ref{prop:ps_compactness} and Proposition \ref{prop:concentration},
there exist points $y_n\in\mathbb R^3$ such that
$\operatorname{dist}(y_n,\mathcal M)\to0$ and, after the appropriate
magnetic gauge correction,
\[
v_n(z) := e^{-iA(y_n)\cdot z} u_n(y_n+\varepsilon_n z) \to w
\]
strongly in $H^1(\mathbb R^3)$, where $w\in S_a$ is an autonomous
ground state satisfying $J_0(w)=E_0(a).$ Using the change of
variables $x=y_n+\varepsilon_n z,$ we obtain
\[
\begin{aligned}
\beta_{\varepsilon_n}(u_n) &= \frac{ \displaystyle \int_{\mathbb
R^3} \chi(y_n+\varepsilon_n z) |v_n(z)|^2\,dz }{ \displaystyle
\int_{\mathbb R^3}|v_n(z)|^2\,dz }.
\end{aligned}
\]
Since $v_n\to w \qquad \text{strongly in }L^2(\mathbb R^3),$ and
since $\chi$ is bounded and uniformly continuous, we have
$\beta_{\varepsilon_n}(u_n) = \chi(y_n)+o(1).$ Moreover,
\[
\operatorname{dist}(y_n,\mathcal M)\to0.
\]
For $n$ sufficiently large,
\[
y_n\in\mathcal M_{\delta/2}\subset B_R(0),
\]
and therefore $\chi(y_n)=y_n.$ Consequently,
\[
\operatorname{dist} \left( \beta_{\varepsilon_n}(u_n), \mathcal M
\right) \to0.
\]
Thus, $\beta_{\varepsilon_n}(u_n) \in\mathcal M_\delta$ for all
sufficiently large $n$, which contradicts the assumption. Therefore,
\[
\beta_\varepsilon \left( \mathcal A_{\varepsilon,\eta} \right)
\subset \mathcal M_\delta
\]
for every sufficiently small $\eta>0$ and $\varepsilon>0$.
\end{proof}
We now construct a continuous map from $\mathcal M$ into the
low-energy set $\mathcal A_{\varepsilon,\eta}$.

Let $w_a$ be the autonomous ground state obtained in Proposition
\ref{prop:limit}. Choose $\zeta\in C_c^\infty(\mathbb R^3,[0,1])$
such that
\[
\zeta(x)=1 \qquad \text{for }|x|\leq\frac{\delta}{2},
\]
and
\[
\zeta(x)=0 \qquad \text{for }|x|\geq\delta.
\]
For $y\in\mathcal M$, define
\[
\widetilde\Psi_{\varepsilon,y}(x) := \zeta(x-y)
w_a\left(\frac{x-y}{\varepsilon}\right) \exp\left(
\frac{i}{\varepsilon} A(y)\cdot(x-y) \right).
\]
Set
\[
\Psi_{\varepsilon,y} := \frac{ a\varepsilon^{3/2} }{
\|\widetilde\Psi_{\varepsilon,y}\|_2 }
\widetilde\Psi_{\varepsilon,y}.
\]
Then $\Psi_{\varepsilon,y} \in S_{a,\varepsilon}.$ Define
\[
\Phi_\varepsilon: \mathcal M\to S_{a,\varepsilon}
\]
by
\[
\Phi_\varepsilon(y) := \Psi_{\varepsilon,y}.
\]

\begin{lemma}[Properties of the localized test map]
\label{lem:test_map} For every $\eta>0$, there exists
$\varepsilon_\eta>0$ such that, for every
$\varepsilon\in(0,\varepsilon_\eta),$ the map
\[
\Phi_\varepsilon: \mathcal M\to\mathcal A_{\varepsilon,\eta}
\]
is well defined and continuous. Moreover,
\[
\lim_{\varepsilon\to0} \sup_{y\in\mathcal M} \left| \frac{
J_{\varepsilon,\tau} (\Phi_\varepsilon(y)) }{ \varepsilon^3 } -
E_0(a) \right| = 0,
\]
and
\[
\lim_{\varepsilon\to0} \sup_{y\in\mathcal M} \left|
\beta_\varepsilon (\Phi_\varepsilon(y)) - y \right| = 0.
\]
\end{lemma}

\begin{proof}
The continuity of $\Phi_\varepsilon$ follows from the continuity of
the translations, the magnetic phase, and the normalization factor.

By Lemma \ref{lem:localized_profiles},
\[
\frac{ J_{\varepsilon,\tau} (\Phi_\varepsilon(y)) }{ \varepsilon^3 }
\to E_0(a)
\]
uniformly for $y\in\mathcal M$. Hence, for every $\eta>0$ and every
sufficiently small $\varepsilon>0$,
\[
J_{\varepsilon,\tau} (\Phi_\varepsilon(y)) \leq \varepsilon^3
\bigl(E_0(a)+\eta\bigr)
\]
for all $y\in\mathcal M$. Thus,
\[
\Phi_\varepsilon(\mathcal M) \subset \mathcal A_{\varepsilon,\eta}.
\]

It remains to study the barycenter. By the change of variables
$x=y+\varepsilon z,$ we obtain
\[
\beta_\varepsilon (\Phi_\varepsilon(y)) = \frac{ \displaystyle
\int_{\mathbb R^3} \chi(y+\varepsilon z) \zeta^2(\varepsilon z)
|w_a(z)|^2\,dz }{ \displaystyle \int_{\mathbb R^3}
\zeta^2(\varepsilon z) |w_a(z)|^2\,dz }.
\]
Since $\zeta(\varepsilon z)\to1$ for every $z\in\mathbb R^3$, and
since $\chi(y+\varepsilon z)\to y$ uniformly with respect to
$y\in\mathcal M$ on bounded subsets of $\mathbb R^3$, the dominated
convergence theorem gives
\[
\beta_\varepsilon (\Phi_\varepsilon(y)) \to y
\]
uniformly for $y\in\mathcal M$.

Therefore,
\[
\lim_{\varepsilon\to0} \sup_{y\in\mathcal M} \left|
\beta_\varepsilon (\Phi_\varepsilon(y)) - y \right| = 0.
\]
The proof is complete.
\end{proof}
We now compare the topology of the low-energy set with that of
$\mathcal M$.

\begin{proposition}\label{prop:topological_relation}
Let $\delta>0$ be sufficiently small. Then there exist
$\eta_\delta>0$ and $\varepsilon_\delta>0$ such that, for every
$\eta\in(0,\eta_\delta)$ and every
$\varepsilon\in(0,\varepsilon_\delta),$ the composition
\[
\beta_\varepsilon\circ\Phi_\varepsilon: \mathcal M\to\mathcal
M_\delta
\]
is homotopic in $\mathcal M_\delta$ to the inclusion map
\[
\iota: \mathcal M\hookrightarrow\mathcal M_\delta.
\]
Consequently,
\[
\operatorname{cat}_{\mathcal A_{\varepsilon,\eta}} \left( \mathcal
A_{\varepsilon,\eta} \right) \geq \operatorname{cat}_{\mathcal
M_\delta}(\mathcal M).
\]
\end{proposition}

\begin{proof}
By Lemma \ref{lem:test_map},
\[
\sup_{y\in\mathcal M} \left| \beta_\varepsilon (\Phi_\varepsilon(y))
- y \right| \to0
\]
as $\varepsilon\to0$.

Hence, for $\varepsilon>0$ sufficiently small, the map
$H:[0,1]\times\mathcal M\to\mathcal M_\delta$ defined by
\[
H(t,y) := (1-t)y + t\, \beta_\varepsilon (\Phi_\varepsilon(y))
\]
is well defined and continuous. Moreover, $H(0,y)=y$ and
\[
H(1,y) = \beta_\varepsilon (\Phi_\varepsilon(y)).
\]
Thus,
\[
\beta_\varepsilon\circ\Phi_\varepsilon \simeq \iota
\]
in $\mathcal M_\delta$. Since
\[
\Phi_\varepsilon: \mathcal M\to \mathcal A_{\varepsilon,\eta}
\]
and
\[
\beta_\varepsilon: \mathcal A_{\varepsilon,\eta} \to \mathcal
M_\delta
\]
are continuous, the standard category inequality yields
\[
\operatorname{cat}_{\mathcal A_{\varepsilon,\eta}} \left( \mathcal
A_{\varepsilon,\eta} \right) \geq \operatorname{cat}_{\mathcal
M_\delta}(\mathcal M).
\]
The proof is complete.
\end{proof}

We now apply the Ljusternik--Schnirelmann theorem to the restriction
of the penalized functional to the prescribed-mass manifold.

\begin{proposition}\label{prop:ls_multiplicity}
Let $a\in(0,a_*)$ be fixed. Then, for every sufficiently small
$\delta>0$, there exist $\eta_\delta>0$ and $\varepsilon_\delta>0$
such that, for every $\eta\in(0,\eta_\delta)$ and every
$\varepsilon\in(0,\varepsilon_\delta),$ the functional
$J_{\varepsilon,\tau} \big|_{S_{a,\varepsilon}}$ possesses at least
$\operatorname{cat}_{\mathcal M_\delta}(\mathcal M)$ distinct
constrained critical points $u_{\varepsilon,1}, \ldots,
u_{\varepsilon,\ell},$ where $\ell \geq \operatorname{cat}_{\mathcal
M_\delta}(\mathcal M),$ and $J_{\varepsilon,\tau}
(u_{\varepsilon,j}) \leq \varepsilon^3 \bigl(E_0(a)+\eta\bigr),
\qquad j=1,\ldots,\ell.$
\end{proposition}

\begin{proof}
Fix $\delta>0$ sufficiently small and choose
$\eta\in(0,\eta_\delta).$ By Proposition
\ref{prop:topological_relation},
\[
\operatorname{cat}_{\mathcal A_{\varepsilon,\eta}} \left( \mathcal
A_{\varepsilon,\eta} \right) \geq \operatorname{cat}_{\mathcal
M_\delta}(\mathcal M).
\]
By Proposition \ref{prop:ps_compactness}, after possibly reducing
$\eta_\delta>0$ and $\varepsilon_\delta>0$, the functional
\[
J_{\varepsilon,\tau} \big|_{S_{a,\varepsilon}}
\]
satisfies the Palais-Smale condition at every level $c \leq
\varepsilon^3 \bigl(E_0(a)+\eta\bigr).$ Therefore, the
Ljusternik--Schnirelmann critical point theorem applied to the
sublevel $\mathcal A_{\varepsilon,\eta}$ yields at least
\[
\operatorname{cat}_{\mathcal A_{\varepsilon,\eta}} \left( \mathcal
A_{\varepsilon,\eta} \right)
\]
distinct constrained critical points of $J_{\varepsilon,\tau}$.

Consequently, $J_{\varepsilon,\tau} \big|_{S_{a,\varepsilon}}$
possesses at least $\operatorname{cat}_{\mathcal M_\delta}(\mathcal
M)$ distinct constrained critical points $u_{\varepsilon,1}, \ldots,
u_{\varepsilon,\ell}$ satisfying
\[
J_{\varepsilon,\tau} (u_{\varepsilon,j}) \leq \varepsilon^3
\bigl(E_0(a)+\eta\bigr).
\]
The proof is complete.
\end{proof}

We finally describe the concentration behavior of the solutions
obtained above.

\begin{lemma}\label{lem:max_concentration}
Let $\varepsilon_n\to0$ and let $u_n\in S_{a,\varepsilon_n}$ be
constrained critical points of $J_{\varepsilon_n,\tau}$ satisfying
\[
J_{\varepsilon_n,\tau}(u_n) \leq \varepsilon_n^3
\bigl(E_0(a)+o(1)\bigr).
\]
Let $x_n$ be a global maximum point of $|u_n|$. Then
$\operatorname{dist}(x_n,\mathcal M)\to0.$ Consequently, $V(x_n)\to
V_0.$
\end{lemma}

\begin{proof}
By Proposition \ref{prop:concentration}, there exist concentration
points $y_n\in\mathbb R^3$ such that
$\operatorname{dist}(y_n,\mathcal M)\to0.$ Moreover, after the
appropriate rescaling and magnetic gauge correction, the sequence
converges strongly to a nontrivial autonomous ground state. By
Proposition \ref{prop:uniform_decay},
\[
\lim_{R\to+\infty} \limsup_{n\to\infty} \sup_{|x-y_n|\geq
R\varepsilon_n} |u_n(x)| = 0.
\]
On the other hand, the nontriviality of the limiting profile implies
that there exists $c_0>0$ such that $\|u_n\|_\infty\geq c_0$ for all
sufficiently large $n$. Therefore, every global maximum point $x_n$
of $|u_n|$ must satisfy $|x_n-y_n| \leq R\varepsilon_n$ for some
$R>0$ independent of $n$. Hence $|x_n-y_n|=O(\varepsilon_n).$ It
follows that
\[
\begin{aligned}
\operatorname{dist}(x_n,\mathcal M) &\leq |x_n-y_n| +
\operatorname{dist}(y_n,\mathcal M)
\\
&\to0.
\end{aligned}
\]
Since $V$ is continuous, $V(x_n)\to V_0.$ The proof is complete.
\end{proof}

\textbf{Proof of Theorem \ref{T2}: } Fix $a\in(0,a_*)$ and choose
$\delta>0$ sufficiently small so that $\mathcal
M_\delta\subset\Lambda.$ Let $\eta>0$ satisfy $0<\eta<
\min\{\eta_\delta,\delta_0\},$ where $\eta_\delta$ is given by
Proposition \ref{prop:ls_multiplicity} and $\delta_0$ is the
low-energy constant appearing in Lemma \ref{lem:penal}.

By Proposition \ref{prop:ls_multiplicity}, for every sufficiently
small $\varepsilon>0$, the penalized functional
$J_{\varepsilon,\tau} \big|_{S_{a,\varepsilon}}$ possesses at least
$\operatorname{cat}_{\mathcal M_\delta}(\mathcal M)$ distinct
constrained critical points $u_{\varepsilon,1}, \ldots,
u_{\varepsilon,\ell},$ where $\ell \geq \operatorname{cat}_{\mathcal
M_\delta}(\mathcal M),$ and
\[
J_{\varepsilon,\tau} (u_{\varepsilon,j}) \leq \varepsilon^3
\bigl(E_0(a)+\eta\bigr)
\]
for every $j=1,\ldots,\ell.$ Since $\eta<\delta_0,$ Lemma
\ref{lem:penal} implies that the penalization is inactive at each
$u_{\varepsilon,j}$. Hence $J_{\varepsilon,\tau} (u_{\varepsilon,j})
= J_\varepsilon (u_{\varepsilon,j})$ and $J_{\varepsilon,\tau}'
(u_{\varepsilon,j}) = J_\varepsilon' (u_{\varepsilon,j}).$
Therefore, every $u_{\varepsilon,j}$ is a constrained critical point
of the original functional $J_\varepsilon$ on $S_{a,\varepsilon}$.

By the Lagrange multiplier rule, for every $j=1,\ldots,\ell$, there
exists $\lambda_{\varepsilon,j}\in\mathbb R$ such that
$J_\varepsilon' (u_{\varepsilon,j}) = \lambda_{\varepsilon,j}
u_{\varepsilon,j}.$ Consequently, problem \eqref{P} possesses at
least $\operatorname{cat}_{\mathcal M_\delta}(\mathcal M)$ distinct
normalized solutions satisfying $\int_{\mathbb R^3}
|u_{\varepsilon,j}|^2\,dx = a^2\varepsilon^3.$ Finally, let
$x_{\varepsilon,j}$ be a global maximum point of
$|u_{\varepsilon,j}|$. By Lemma \ref{lem:max_concentration},
\[
\operatorname{dist} (x_{\varepsilon,j},\mathcal M) \to0 \qquad
\text{as }\varepsilon\to0.
\]
Therefore, $V(x_{\varepsilon,j}) \to V_0.$ Thus, the solutions
concentrate around the minimum set $\mathcal M$ in the semiclassical
limit. The proof is complete. \qed

\section{Conclusion}

In this paper, we have established the existence and multiplicity of
normalized semiclassical solutions for a magnetic
Schr\"odinger--Poisson system with critical local and nonlocal
nonlinearities. By combining constrained variational methods, a
penalization scheme, concentration--compactness arguments, and
Ljusternik--Schnirelmann theory, we obtained at least one normalized
solution for small mass and small semiclassical parameter. We also
proved a multiplicity result governed by the topology of the minimum
set $\mathcal M$ of the electric potential. Finally, the obtained
solutions were shown to concentrate near $\mathcal M$ as
$\varepsilon\to0$.

\bigskip

%%%%%%%%%%%%%%%%%%%%%%%%%%%%%%%%%%%%%%%%%%%%%%%%%%%%%%%%%%%%%%%%%%%%%%%%%%%%%%%%%%%%%%%%%%%%%%%%%%%%%%%%%%%%%%%%%%%%%%%%%%%%%%%%%%%%%%%%%%%%%%%%%%%%%%%%%%%%%%%%%%%%%%%%%%%%%%%%%%%%%%%%%%%%%%%%%%%%%%%%%%%%%%%%%%%%%%%%%%%%%%%%%%%%%%%%%%%%%%%%%%

\noindent \textsc{$^{1}$ Khaled Khachnaoui}\\
$^{1}$ University of Kairouan, Preparatory Institute for Engineering
Studies \\Department of
 Mathematics\\Tunisia,\\
\texttt{k{\_}khachnaoui@yahoo.com} \\
\noindent \mbox{} \\

\end{document}